\documentclass[10pt]{amsart}
\usepackage{amsthm,amsmath,amssymb,amscd,graphics,enumerate, stmaryrd,xspace,verbatim, epic, eepic,url}
\usepackage[usenames,dvipsnames]{color}
\usepackage[active]{srcltx}
\usepackage[all]{xypic}

\usepackage{defs}

\def\comp{{\rm comp}}
\def\sat{{\rm sat}}
\def\nor{{\rm nor}}
\def\out{{\rm out}}
\def\inn{{\rm in}}
\def\Bsch{{\rm Bsch}}
\def\hatBVar{{\widehat{\rm BVar}}}
\def\smal{{\rm small}}
\def\ssing{{\rm ssing}}
\def\sreg{{\rm sreg}}
\def\BVar{{\rm BVar}}
\def\princ{{\rm princ}}
\def\emb{{\rm emb}}

\def\cosupp{{\rm cosupp}}

\def\hatVar{{\mathrm {\wh{Var}}}}

\def\sqord{\coprod_{ord}}

\begin{document}

\title[{Functorial desingularization over $\bfQ$: the embedded case}]{Functorial desingularization over $\bfQ$: boundaries and the embedded case}
\author{Michael Temkin}
\thanks{I am grateful to E. Bierstone and P. Milman for helpful discussions, to V. Berkovich for encouragement and showing his work \cite{berform}, and to the anonymous referee for pointing out numerous typos. This work was supported by the Israel Science Foundation (grant No. 1159/15). The first version was written when the author stayed at the Institute for Advanced Study and was supported by NSF grant DMS-0635607.}
\address{Einstein Institute of Mathematics, The Hebrew University of Jerusalem, Giv'at Ram, Jerusalem, 91904, Israel}
\email{temkin@math.huji.ac.il}


\maketitle

\begin{abstract}
Our main result establishes functorial desingularization of noetherian quasi-excellent schemes over $\bfQ$ with ordered boundaries. A functorial embedded desingularization of quasi-excellent schemes of characteristic zero is deduced. Furthermore, a standard simple argument extends these results to other categories, including in particular, (equivariant) embedded desingularization of the following objects of characteristic zero: qe algebraic stacks, qe formal schemes, complex and non-archimedean analytic spaces. We also obtain a semistable reduction theorem for formal schemes.
\end{abstract}

\section{Introduction}
Very often one divides various desingularization problems into two large classes: non-embedded desingularization and embedded desingularization. A typical example of a problem of the first type is to associate to a scheme $X$ a blow up sequence $f\:X'\longto X$ with a regular $X'$ and such that $f$ is an isomorphism over the regular locus of $X$. A typical (but rather crude) example of a problem of the second type is to associate to a regular ambient scheme $X$ with a divisor $Z\into X$ a blow up sequence $f\:X'\longto X$ such that $f^{-1}(Z)$ is an snc divisor and $f$ blows up only regular subschemes in the preimage of $Z$. In particular, $X'$ is regular and $f$ is an isomorphism over $X\setminus Z$. Although there are much finer versions of embedded desingularization, it seems that the one we have mentioned covers most of the applications of embedded desingularization.

This work is a direct continuation of \cite{nemb}, where functorial non-embedded desingularization of varieties of characteristic zero was used to prove an analogous result for all quasi-excellent schemes of characteristic zero. Our aim is to extend the technique developed in \cite{nemb} to functorial embedded desingularization of qe schemes of characteristic zero. In particular, we establish the above version of embedded desingularization for such schemes. Note that in view of \cite[$\rm IV_2$, \S7.9]{ega}, this is the most general class of schemes over $\bfQ$ for which the problem can be solved. Moreover, we solve a finer problem formulated in Theorem \ref{embth}, though when compared with the case of varieties, this is still far from the strongest known version. In particular, we do not achieve principalization and our algorithm does not choose centers that have simple normal crossings with exceptional divisors.

As a simple corollary of functorial embedded desingularization we deduce analogous embedded desingularization results for (formal) qe stacks, and complex/non-archimedean analytic spaces. Also, one obtains equivariant embedded desingularization of all these objects with respect to an action of a regular group.

\subsection{Main results}\label{mainsec} Now, we are going to formulate our main results. We try to make the formulations as self-contained as possible, though certain referencing to the terminology introduced later is still involved. Mainly, one has to use the notions of principal and complete transforms of the boundary in order to formulate the sharpest results. Our results are concerned with desingularization of the following objects: a divisor on a regular scheme, a scheme with a boundary, a scheme embedded into a regular scheme with an snc boundary. Although using the non-embedded desingularization from \cite{nemb} all three results can be easily obtained one from another, we decided to formulate them all since each of them has its own flavor. The first and the third cases are the classical embedded desingularization problems. The second formulation is important for us because the entire paper is written in the language of $B$-schemes, i.e. schemes with boundaries. The reasons for choosing this language is discussed in \S\ref{bsubsec} and Appendix \ref{bmotsec}.

\subsubsection{Desingularization of divisors}
In many applications of embedded desingularization one wants to resolve a divisor $E$ (or a function) on a regular ambient variety $X$ by finding a modification $f\:X'\to X$ with a regular $X'$ and such that the reduction of $E\times_XX'$ is snc (i.e. strictly normal crossings). One can also achieve that $f$ modifies only the non-regular locus of $E$, but the problem of preserving the entire locus where $E$ is snc is more delicate, and, probably, is not the ``correct problem" (see \S\ref{strsec}). In fact, the problem that makes much more sense is to find a desingularization that preserves the locus where $E$ is snc and the splitting to components is fixed in some sense. A standard way to formulate this is to consider a {\em divisorial boundary} $E=\{E_1\. E_n\}$ where each $E_i$ is a divisor.

\begin{rem}
(i) We make our life easier by considering only {\em ordered} boundaries, as otherwise only the new boundary would be ordered accordingly to the history, and this would require to use a heavier terminology, as one does in \cite{CJS}. In particular, we will use the order whenever this shortens our arguments. As a drawback, our desingularization procedure depends on the order of the components and is only compatible with regular morphisms that preserve the order.

(ii) At least in the case of varieties one can functorially desingularize unordered boundaries so that the entire snc locus is preserved, see Remark \ref{strrem}(iii). So, almost surely all our results have ``unordered" analogs, where all initial boundaries are unordered.
\end{rem}

We refer to \S\ref{bschsec} for the definitions of snc and strictly monomial boundaries, and to \S\ref{blsec} for the definition of complete and principal transforms of a boundary under a blow up sequence. For the reader's convenience basic properties of the transforms are collected in Lemma \ref{sumlem}. Recall that if $E=\{E_1\.E_n\}$ is a boundary then $|E|=\cup_{i=1}^n|E_i|$ denotes its support and $[E]=\sum_{i=1}^nE_i$ denotes its scheme-theoretic support.

\begin{theor}\label{divth}
For any quasi-excellent noetherian regular scheme $X$ of characteristic zero and a divisorial boundary $E$ on $X$ there exists a blow up sequence $f=\calF_\div(X,E)\:X'\longto X$ such that

(i) the centers of $f$ are regular (and so $X'$ is regular) and contained in the preimage of the set $T$ consisting of the points $x\in E$ at which $E$ is not snc,

(ii) the complete transform $f^\circ(E)$ is snc; in particular, the strict transform $f^!(|E|)$ is snc and the total transform $[E]\times_XX'$ is strictly monomial,

(iii) $\calF_\div$ is functorial with respect to strict regular morphisms (cf. \S\ref{Bmorsec}); that is, given a regular morphism $g\:Y\to X$ and $D=E\times_XY$, the blow up sequence $\calF_\div(Y,D)$ is obtained from $g^*(\calF_\div(X,E))$ by omitting all empty blow ups.
\end{theor}

\begin{rem}
(i) Unlike the known algorithms for varieties, we do not achieve that the centers of $f$ are transversal to the new boundaries.

(ii) The algorithm only modifies the bad locus $T$, but its blow ups $X_{i+1}\to X_i$ can modify the intermediate good loci $X_i\setminus T_i$ where $E_i$ is snc. More generally, in all desingularization algorithms in this paper, we require that the centers lie over the bad locus of the original scheme $X$ (with an additional datum), but they can intersect the intermediate good loci where $X_i$ is already resolved.
\end{rem}

\subsubsection{Desingularization of $B$-schemes}
One does not have to require that $X$ is regular in the above theorem. Also, it is convenient to allow non-divisorial boundaries $B=\{B_1\. B_n\}$, where each component $B_i$ is only a locally principal closed subscheme. A pair $(X,B)$ will be called a {\em $B$-scheme}, and we will work within the framework of $B$-schemes in this paper. In particular, the version of the main theorem we will deal with in the paper is given below. In this theorem, a $B$-scheme $(X',B')$ is said to be {\em semi-regular} if $X'$ is regular and $B'$ becomes snc after removing connected components of $X'$ from the boundary components $B'_i\in B'$.

\begin{theor}\label{Bth}
For any quasi-excellent noetherian $B$-scheme $(X,B)$ of characteristic zero there exists a blow up sequence $f=\calF(X,B)\:X'\longto X$ such that

(i) $(X',B')$ is semi-regular where $B'=f^\circ(B)$,

(ii) each center of $f$ is regular and disjoint from the preimage of the semi-regular locus of $(X,B)$,

(iii) $\calF$ is functorial with respect to strict regular morphisms; that is, given a regular morphism $g\:Y\to X$ with $D=B\times_XY$, the blow up sequence $\calF(Y,D)$ is obtained from $g^*(\calF(X,B))$ by omitting all empty blow
ups.
\end{theor}

\begin{rem}\label{brem}
(i) Up to the non-embedded desingularization theorem \cite[1.2.1]{nemb}, Theorems \ref{divth} and \ref{Bth} are equivalent because one can construct $\calF(X,B)$ as the composition of the non-embedded desingularization $g\:X'\longto X$ with the blow up sequence $\calF_\div(X',g^\circ(B)^\div)\:X''\longto X'$, where $g^\circ(B)^\div$ is the divisorial part of $g^\circ(B)$.

(ii) We do not require that $X$ is generically reduced, and the algorithm simply blows up the non-reduced components at some stage.
\end{rem}

\subsubsection{Embedded desingularization}
Here is the strongest version of embedded desingularization which is achieved by our method so far, and we will see in \S\ref{proofsubsec} that it follows easily from Theorem \ref{Bth}. The main weakness of this variant is that it does not provide strong principalization in the sense of \S\ref{princsec} below. In addition, the center of the $i$-th blow up $X_i\to X_{i-1}$ does not have to be transversal to the boundary $E_{i-1}$, which is the complete transform of $E_0=E$. In particular, the intermediate boundaries $E_i$ can be singular, though the final boundary $E_n$ is snc.

\begin{theor}\label{embth}
For any quasi-excellent regular noetherian scheme $X$ of characteristic zero with an snc boundary $E$ and a closed subscheme $Z\into X$ there exists a blow up sequence $f=\calF_\emb(X,E,Z)\:X'\longto X$ such that

(i) $X'$ is regular, $E'=f^\circ(E)$ is snc and $Z'=f^!(Z)$ is regular and has simple normal crossings with $E'$,

(ii) each center of $f$ is regular and contained in the preimage of the set $T$ consisting of points $x\in Z$ such that $Z$ has not simple normal crossings with $E$ at $x$ (e.g. $Z$ is not regular at $x$),

(iii) $\calF_\emb$ is functorial with respect to strict regular morphisms; that is, given a regular morphism $g\:Y\to X$ with $D=E\times_XY$ and $W=Z\times_XY$, the blow up sequence $\calF_\emb(Y,D,W)$ is obtained from $g^*(\calF_\emb(X,E,Z))$ by omitting all empty blow ups.
\end{theor}

\subsubsection{Principalization}\label{princsec}
In the case of varieties, one can strengthen the above theorem by adding the condition that the principal transform $f^\rhd(Z)$ equals to the strict transform $Z'=f^!(Z)$. In particular, this implies that $Z\times_XX'=Z'+E_Z$ where $E_Z$ is a strictly monomial exceptional divisor, and after composing with the blow up along $Z'$ one obtains a strong {\em principalization} of $Z$ by a blow up sequence $X''\longto X$. Namely, $Z''=Z\times_XX''$ is a divisor with support on the snc divisor $E''$, and hence $Z''$ is strictly monomial. In this paper we only establish a weaker principalization for qe schemes over $\bfQ$, which obviously follows from our other results and suffices for many applications.

\begin{theor} \label{princth}
For any quasi-excellent noetherian scheme $X$ of characteristic zero with a closed subscheme $Z\into X$ there exists a $(Z\cup X_\sing)$-supported blow up sequence $\calF_\princ(X,Z)\:X'\longto X$ such that $X'$ is regular, $Z\times_XX'$ is strictly monomial and $\calF_\princ$ is functorial with respect to strict regular morphisms. \end{theor}
\begin{proof}[Proof (assuming Theorem \ref{Bth})]
Let $f\:X'\to X$ be the blow up along $Z$ and $Z'=Z\times_XX'$. Note that $B'=\{Z'\}$ is a boundary on $X'$ and consider the desingularization $g=\calF(X',B')\:(X'',B'')\longto(X',B')$ of the $B$-scheme $(X',B')$ as in Theorem~\ref{Bth}. Note that $g$ is $Z$-supported because the bad locus of $(X',B')$ sits over $Z$. Also, $Z\times_XX''$ is a divisor with support contained in the snc boundary $B''$ and hence it is strictly monomial. We define $\calF_\princ(X,B)$ to be the composition $X''\longto X$.
\end{proof}

\subsubsection{Other categories}
Using the same argument as in \cite[\S5]{nemb} one can use the main desingularization theorems for noetherian qe schemes to prove their analogs for other (quasi-compact or not) geometric objects of characteristic zero. Also, it follows from the functoriality that the obtained desingularizations are equivariant.

\begin{theor}\label{catth}
(i) The functors $\calF$, $\calF_\div$, $\calF_\emb$ and $\calF_\princ$ induce analogous functors for quasi-compact (formal) qe stacks and complex/non-archimedean analytic spaces of characteristic zero.

(ii) All these functors extend to non-quasi-compact objects at the cost of replacing blow up sequences with blow up hypersequences in the sense of \cite[\S5.3.2]{nemb}, or simply with a proper desingularization morphism $X'\to X$ without a blow up sequence structure.

(iii) All these desingularizations are equivariant with respect to an action of a smooth group.
\end{theor}

\subsubsection{Compactification of a regular scheme with an snc divisor}
One often uses embedded resolution of singularities to compactify a smooth variety by adding an snc divisor. One may also want to compactify a variety containing an snc divisor. For example, this might be the case when one considers a partial compactification (e.g. by a horizontal divisor) that one wants to extend to a full compactification. Here is the corresponding result in the context of general qe schemes.

\begin{theor}\label{compacth}
Assume that $S$ is a noetherian qe scheme of characteristic zero, $Y$ is a regular scheme equipped with a separated morphism $Y\to S$ of finite type, and $Y_0\into Y$ is an open subscheme such that $D=Y\setminus Y_0$ is an snc divisor. Then there exists a regular scheme $\oY$, a proper morphism $\oY\to S$, and an open immersion $Y\into\oY$ over $S$ such that $\oD=\oY\setminus Y_0$ is an snc divisor.
\end{theor}

\begin{rem}
(i) Theorem \ref{compacth} looks very natural or even typical, and we will see that it is an easy corollary of Theorem~\ref{embth}. In a private correspondence, the author suggested to N. Solomon to use this result in his thesis in order to construct certain compactifications. To our surprise, we could not find this result in the existing literature, even when $Y$ is a $k$-variety and $S=\Spec(k)$.

(ii) Another surprise with Theorem \ref{compacth} is that a similar statement fails when $D$ is only assumed to be a normal crossings divisor, see Example~\ref{compexam}(ii). Perhaps this explains why Theorem~\ref{compacth} is, probably, new.
\end{rem}

\subsubsection{Semistable reduction for formal varieties}
Let $R$ be a discrete valuation ring of residue characteristic zero. Hironaka's theorem provides desingularization of $R$-schemes of finite type, and by purely combinatorial tools one deduces from it the semistable reduction theorem for generically smooth $R$-schemes of finite type, see \cite[p.198]{KKMS}. As an application of resolution of formal schemes, we will prove an analogous theorem for formal $R$-schemes. We refer to \cite[Section~2.1]{temdes} and \cite[\S\S2.4.9-2.4.11]{nemb} for the terminology on formal schemes and their blow ups and singular loci. See also Section~\ref{dregsec} for the definition of semistability. By the union of closed formal subschemes we mean the closed formal subscheme given by the intersection of the corresponding ideals. Recall that a morphism is {\em of topologically finite type} if it is locally of the form $\Spf(A)\to\Spf(R)$, where $R$ is an adic ring with an ideal of definition $J$ and $$A=R\{t_1\.t_n\}\llbracket x_1\.x_m\rrbracket/I$$ with an ideal of definition $(J,x_1\.x_n)$.

\begin{theor}\label{stabth}
Assume that $R$ is a complete discrete valuation ring containing $\bfQ$ and $\gtX$ is a reduced formal scheme, flat and of topologically finite type over $\gtS=\Spf(R)$. Let $k$ be the residue field of $R$ and $\gtZ=\gtX\otimes_Rk$. Then there exists a $(\gtX_\sing\cup\gtZ)$-supported blow up $\gtf\:\tilgtX\to\gtX$ and a finite extension of discrete valuation rings $R'/R$ with $\gtS'=\Spf(R')$ such that the normalization $\gtX'$ of $\tilgtX\times_\gtS\gtS'$ is regular and strictly semistable over $\gtS'$.
\end{theor}

\begin{rem}\label{stabrem}
(i) Similarly to \cite[Theorem~2.1.2(ii)]{berform}, we do not restrict semistable reduction to formal schemes of finite type over $\gtS$.

(ii) Note that $\gtZ$ is a closed formal subscheme of $\gtX$ that contains the closed fiber $\gtX_s$ and is called in \cite{bervanish3} the special fiber of $\gtX$ over $\gtS$. If $\gtX$ is rig-regular (see \S\ref{rigsec}) then $\gtZ$ contains $\gtX_\sing$ and therefore $\gtf$ is $\gtZ$-supported.
\end{rem}

\subsection{Overview}
Now, let us briefly discuss the structure of the paper.

\subsubsection{B-schemes}\label{bsubsec}
We devote \S\ref{bsec} to the definition and basic study of $B$-schemes, $B$-blow ups, desingularization of $B$-schemes, etc. In particular, we define boundaries and their transforms under blow up sequences. Such objects naturally arise in the desingularization theory, though they were introduced only very recently in \cite{CJS}. Actually, if one simply restricts an snc boundary $E$ on a regular scheme $X$ to a closed subscheme $Z\into X$, then the restriction $E|_Z$ is a (not necessarily divisorial) boundary. We make one further step with respect to \cite{CJS} by combining schemes and boundaries into a single object, a $B$-scheme.

A partial justification for introducing this new notion is that, up to the order of the boundary, $B$-schemes admit a nice interpretation as Zariski log schemes $(X,M)$ such that all the stalks $\oM_x$ are free monoids $\bfN^{k(x)}$. This observation is not really used in the paper but might be instructive.

\begin{rem}\label{bmotrem}
(i) Classical embedded desingularization operates with a datum $(X,E,Z)$. The $B$-scheme $(Z,E|_Z)$ provides a more economical but essentially equivalent desingularization datum. Our method of extending desingularization algorithms from varieties to qe schemes passes through formal varieties and involves algebraization of rig-regular formal varieties $\gtZ$ via Elkik's theorem. Recently Trushin extended this theorem to couples $(\gtZ,\gtD)$, where $\gtD$ is a rig-smooth closed formal subvariety of $\gtZ$, see \cite[Theorem~51]{Trushin}. Perhaps, this can be extended further to rig-snc divisors, triples of subschemes, etc., but this is not known. Current algebraization results do not allow to work with formal triples $(\gtX,\gtE,\gtZ)$, but we manage to apply them to certain formal $B$-varieties. This is the main reason for using $B$-schemes in this paper. Some other motivation is discussed in the Appendix.

(ii) The final decision to adopt the language of $B$-schemes was made when I saw \cite{CJS} and its technique of working with general boundaries. In particular, the notions of principal and complete transforms are borrowed from \cite{CJS}.
\end{rem}

\subsubsection{Desingularization functors}
In \S\ref{dessec} we prove our main results on desingularization of $B$-schemes. This is done in four steps worked out in \S\S\ref{bvarsec}--\ref{genbsec}. In \S\ref{bvarsec} we establish the case of $B$-varieties by constructing the functor $\calF_\BVar$. The general idea is to simply apply the non-embedded desingularization to $X$ and then apply the embedded desingularization to the boundary. However, one must be slightly more careful in order not to destroy the entire snc locus of $B$, and for this we add an intermediate step in which we separate the old boundary from the singular locus.

Then, we extend in \S\ref{formvarsec} the functor $\calF_\BVar$ to the functor $\hatcalF_\BVar$ on formal $B$-varieties whose boundary and singular locus are supported on the closed fiber. The analogous step in \cite{nemb} is the most technical and subtle one. Fortunately, the argument from \cite{nemb} extends verbatim to our more general situation.

In \S\ref{smallsec} we desingularize a $B$-scheme $(X,B)$ with a fixed divisor $Z\into X$ which contains the bad locus $(X,B)_\ssing$ and is a disjoint union of varieties. The first and main step is to separate the old boundary from the bad locus, and this is done in the same way as in the case of varieties. After that the formal completion of $(X,B)$ along $Z$ can be desingularized by $\hatcalF_\BVar$. Moreover, the latter desingularization blows up only open ideals and hence algebraizes to a desingularization of $(X,B)$.

Finally, in \S\ref{genbsec} we construct a desingularization $\calF(X,B)$ of general qe $B$-schemes $(X,B)$ of characteristic zero. This is based on the desingularization of $B$-schemes with small bad locus and is done by induction on codimension similarly to the proof of \cite[Theorem~1.2.1]{nemb} in \cite[\S4.3]{nemb}.

\subsubsection{Semistable reduction}
In Section \ref{semisec} we deduce semistable reduction for formal schemes, see Theorem~\ref{stabth}, from desingularization of formal varieties with divisors, see Theorem~\ref{catth}. As in the case of schemes, such a reduction can be done by combinatorial methods of \cite{KKMS}. Usually this phrase is the only argument one provides in analogous situations, and it is not easy to find details in the existing literature. Section~\ref{semisec} aims to fill this gap. In addition, we use the language of Kato's fans and log regular schemes, that were developed after \cite{KKMS}, and we show how log regularity extends to qe formal schemes.

In fact, we reduce the semistable reduction to what we call a $d$-regularization problem, see Theorem~\ref{dregth}. It is a more general absolute statement, that makes sense for any qe formal scheme with a divisor. Applying Theorem~\ref{catth}, the latter result reduces to the particular case of log regular formal log schemes with monoidal divisors, see Theorem~\ref{formmonprth}. Finally, we use Kato's fans to reduce $d$-regularization on log regular (formal) log schemes to the main combinatorial result of \cite{KKMS}.

\subsubsection{The Appendix}
One could write this paper without using the notion of $B$-schemes and even without using not snc boundaries; of course, this would require to formulate the main results in another but equivalent way. Nevertheless, the language of $B$-schemes seems to be very natural for our task and we discuss the reasons for this in the Appendix. It is not used in the paper, but can be instructive. In particular, we explain in \S\ref{strsec} why a naive boundary, which is a single divisor, would not work as fine as our notion, and also correct a mistake in \cite{temdes} which was caused by a confusion between these two.

\subsubsection{Our method and future research}
The bottleneck of the method used in this paper and in \cite{nemb} is the passage from varieties to formal varieties. It makes use of Elkik's algebraization theorem, which does not apply to complicated desingularization data. Trushin's theorem \cite[Theorem~51]{Trushin} improves the situation slightly. In fact, using it one could strengthen the results of this paper in few aspects. In particular, one could establish $B$-strong desingularization of generically reduced $B$-schemes (see Remarks~\ref{BVarrem} and \ref{formrem}). This strengthening is not used in most of applications but would require a substantial additional work, so we decided not to work it out in the paper.

Furthermore, currently it seems that a simpler approach to desingularization of qe schemes is to directly show that the algorithm of Bierstone-Milman applies to arbitrary qe schemes of characteristic zero that possess enough derivations. This would cover the case of formal varieties and would make the algebraization step unnecessary. I plan to study this approach elsewhere.

\subsubsection{Conventions}\label{convsec}
All (formal) schemes are assumed to be locally noetherian. A {\em variety} means a scheme of finite type over a field. We keep the conventiones of \cite[\S2]{nemb}; in particular, a blow up sequence ``remembers" the centers of all blow ups, and a blow up is called {\em trivial} or {\em empty} if its center is empty.

By a {\em component} of a scheme $X$ we mean a disjoint union of several connected components of $X$. Assume that $X$ is a scheme with a closed subscheme $Z$. Then by $|Z|$ we denote the support of $Z$, which is the underlying closed set, and by $\calI_Z\subseteq\calO_X$ we denote the ideal of $Z$. So, $Z=\Spec_X(\calO_X/\calI_Z)$. We say that $Z$ is {\em locally principal} (resp. a {\em Cartier divisor}) if $\calI_Z$ is locally principal (resp. invertible). If $D\into X$ is a Cartier divisor then for any $n\in\bfN$ we define $Z+nD$ to be the closed subscheme defined by $\calI_Z\calI^n_D$. Note also that the fractional ideal $\calI_Z\calI^{-n}_D$ is an ideal if and only if $nD\into Z$, and in this case we denote the corresponding subscheme as $Z-nD$.

Given a morphism $f\:X'\to X$, it will be convenient to use the notation $f^*(Z):=Z\times_XX'$ for the pullback of $Z$, and when $f$ is an immersion, we will often call $f^*(Z)$ the {\em restriction} of $Z$ onto $X'$ and denote it by $Z|_{X'}$. Also, in this case for any morphism $g\:Y\to X$ (e.g. a blow up) we will write $g|_{X'}=g\times_XX'$.

%\tableofcontents

\section{Boundaries and desingularization}\label{bsec}

\subsection{Schemes with boundaries}\label{bschsec}

\subsubsection{Boundary}
In this paper we will only work with ordered boundaries so by a {\em boundary} on a scheme $X$ we mean a tuple $B=\{B_i\}_{i\in I}$ indexed by a finite ordered set $I$ in which each $B_i$ is a locally principal subscheme of $X$. It can happen that $B_i=B_j$ for $i\neq j$. Each $B_i$ is called a {\em component} of $B$ or a {\em boundary component}. We say that $B$ is {\em divisorial} if all its components are Cartier divisors.

We will ignore empty components and will be only interested in the equivalence class of the ordered index set $I$. So, any boundary can be uniquely represented in its {\em reduced form} as $B=\{B_1\. B_n\}$, where all $B_i$'s are non-empty. We say that two boundaries on $X$ are {\em equal} if their reduced forms are equal. Nevertheless, it is convenient to also consider boundaries in non-reduced form because they can be produced by natural operations. For example, some components can vanish when restricting onto a closed subscheme $X'$ of $X$.

We define the {\em ordered disjoint union} of boundaries as $$B\sqord B'=\{B_1\. B_n,B'_1\. B'_{n'}\}.$$ The {\em support} of $B$ is the closed subset $|B|=\cup_{i\in I}|B_i|$ of $X$, and we also define a finer {\em schematical support} of $B$ as $[B]=\sum_{i\in I}B_i$. It is well defined even for non-divisorial boundaries.

\begin{rem}
(i) An analogous definition of boundaries is given in \cite{CJS}, where one prefers the non-ordered variant of the definition.

(ii) We do not require that the elements that locally define $B_i$'s are not zero divisors, so $B_i$'s do not have to be Cartier divisors. This is convenient because we can then restrict a boundary onto any closed subscheme.
\end{rem}

\subsubsection{The stratification induced by $B$}
For each point $x\in X$ by $I(x)$ (resp. $\tilI(X)$) we denote the set of all $i\in I$ such that $x\in B_i$ (resp. $x\in B_i$ and $B_i$ is not the whole $X$ in a neighborhood of $x$, i.e. the element defining $B_i$ does not vanish in $\calO_{X,x}$). Also, for any subset $J=\{j_1\. j_m\}\subseteq I$ we set $\oB_J:=B_{j_1}\times_X\dots\times_X B_{j_m}$ and define the $J$-th stratum $B_J$ of $B$ as the open subscheme of $\oB_J$ obtained by removing each $\oB_{J'}$ with $J\subsetneq J'$. In particular, $\oB_\emptyset=X$ and $x\in B_{I(x)}$.

\begin{rem}
Although we will not need the following observation in this paper, it can give an alternative point of view on the nature of $B$. Giving an unordered boundary on $X$ is equivalent to giving a Zariski log structure $M$ on $X$ such that for each point $x\in X$ the monoid $\oM_x$ is free. Actually, under this correspondence one has that $\oM_x\toisom\bfN^{I(x)}$ and the images of the generators of $\oM_x$ in $\calO_X$, which are well defined up to units, define the subschemes $B_i$ locally at $x$.
\end{rem}

\subsubsection{Pullback and restriction}
Given a morphism $f\:Y\to X$ and a boundary $B=\{B_i\}_{i\in I}$ on $X$, the pullbacks $f^*(B_i):=B_i\times_XY$ are locally principal and we define the pullback boundary $f^*(B)=\{f^*(B_i)\}_{i\in I}$. In the case when $f$ is an immersion we will also call $f^*(B)$ the {\em restriction} of $B$ on $Y$ and denote it $B|_Y$.

\subsubsection{$B$-schemes}
A {\em scheme with boundary} or simply a {\em $B$-scheme} is a pair $(X,B)$ consisting of a scheme $X$ with a boundary $B$. We will say that $(X,B)$ is {\em qe, of characteristic zero, generically reduced}, etc. if the scheme $X$ is so. Note, however, that the notion of regular $B$-schemes will be defined below in a different way.

\subsubsection{Morphisms of $B$-schemes}\label{Bmorsec}
A morphism of $B$-schemes $F:(X',B')\to(X,B)$ is a morphism of the corresponding log schemes $(X',\calM_{X'})\to(X,\calM_{X})$, i.e a morphism $f\:X'\to X$ and a homomorphism $f_\calM\:f^*\calM_X\to\calM_{X'}$ compatible with the structure homomorphisms $\calM_X\to\calO_X$ and $\calM_{X'}\to\calO_{X'}$ of the log schemes. Locally, giving such a homomorphism $f_\calM$ is equivalent to giving a matrix $g\in\bfN^{I\times I'}$ such that $f^*(B_i)=gB'=\sum_{i'\in I'}g(i,i')B'_{i'}$ for each $i\in I$. In fact, if $B_i$ are connected then $g$ is also defined globally. We say that $F$ is {\em regular} if $f$ is regular.

If $g$ is the identity matrix then we say that $F$ is {\em strict}. Thus, strict morphisms of $B$-schemes correspond to strict morphisms of log schemes (i.e. pairs $f,f_\calM$ such that $\calM$ is an isomorphism) that also respect the order of components. Given a $B$-scheme $(X,B)$, any morphism $f\:X'\to X$ extends to a strict morphism $F$ in a unique way by taking $B'=f^*(B)$.

\begin{rem}
(i) The reader can ignore the definition involving log geometry. In this paper we will only be interested in strict regular morphisms and $B$-blow up sequences, see \S\ref{blsec}. The latter can (and will) be introduced in a simple ad hoc manner; they are almost never strict.

(ii) Note, for sake of completeness, that in ``non-degenerate'' situation $g$ (and $f_\calM$) is determined by $f$. For example, this is the case then all components $B'_i$ are Cartier divisors without common irreducible components. However, already when $X'=X$, $B=\{B_1\}$ and $B'=\{B'_1,B'_2\}$ with $B'_1=B'_2=B_1$, there are two choices of $g$.
\end{rem}

\subsubsection{Snc and monomial boundaries}
A boundary $B$ on $X$ is called {\em snc at $x$} if $B_{I(x)}$ is regular and of codimension at least $|I(x)|$ at $x$. We say that $B$ is {\em snc} if it is snc at any point of $X$, in particular, our definition is not local at $|B|$. Similarly, we say that $B$ is {\em strictly monomial} if $X$ is regular and $|B|$ is an snc divisor. Snc boundaries satisfy the following nice properties, that are often used to define snc divisors.

\begin{lem}\label{snclem}
If $B$ is snc at $x$ then each $B_i$ is divisorial at $x$ and for any $J\subseteq I(x)$ the closed stratum $\oB_J$ is regular and of codimension $|J|$ at $x$. In particular, $B$ is snc if and only if each non-empty $\oB_J$, including $\oB_\emptyset=X$, is regular of pure codimension $|J|$.
\end{lem}
\begin{proof}
The codimension of a locally principal subscheme does not exceed one, hence the codimension of $B_{I(x)}$ at $x$ is precisely $|I(x)|$, and it follows that for each $J\subseteq I(x)$ the codimension of $\oB_J$ at $x$ is precisely $|J|$. Now, the regularity of $\oB_J$ at $x$ follows from the following simple observation: if $A$ is a local noetherian ring with an element $x\in A$ such that $\Spec(A/xA)$ is regular of codimension 1 in $\Spec(A)$ then $A$ is regular and $x$ is not a zero divisor in $A$.
\end{proof}

For completeness we also discuss a connection between different notions. This simple result will not be used, so the proof is omitted.

\begin{lem}\label{sncrem}
For a boundary $B$ the following conditions are equivalent:

(i) $B$ is snc,

(ii) the closed subscheme $[B]$ is an snc divisor, and all components $B_i$ of $B$ are regular,

(iii) $B$ is strictly monomial, each $B_i$ is a regular divisor (not necessarily connected) and no pair $B_i,B_j$ with
$i\neq j$ has a common irreducible component.
\end{lem}

\subsubsection{Regular and semi-regular $B$-schemes}
A $B$-scheme $(X,B)$ is {\em regular} if $B$ is snc. Also, we will often use a slight weakening of the regularity condition. Namely, we say that a $B$-scheme $(X,B)$ is {\em semi-regular} if locally at each point $x\in X$ the stratum $B_{\tilI(x)}$ is regular and of codimension $|\tilI(x)|$. Semi-regularity at $x$ means that for a neighborhood $U$ of $x$ we can split $B|_U$ as unordered disjoint union $B'\coprod B''$ so that $B'$ is snc and $B''$ consists of a few copies of $U$.

\subsubsection{Regular and singular locus}
For a qe $B$-scheme $(X,B)$ the set of points $x$ at which $(X,B)$ is semi-regular form an open subset which will be denoted $(X,B)_\sreg$. Its complement will be denoted $(X,B)_\ssing$ and will often be simply called the {\em bad locus} of $(X,B)$. (Another option might be the ``strongly singular" locus.) We will not make use of other regular/singular loci of $(X,B)$.

\begin{lem}\label{locuslem}
Assume that $X$ is a scheme, $B=\{B_1\.B_n\}$ is a divisorial boundary, and $1\le l\le n$. Set $\tilX=B_l$, $B-B_l=\{B_1\.B_{l-1},B_{l+1}\. B_n\}$, and $\tilB=(B-B_l)|_\tilX$. Then $(X,B)_\ssing\cap\tilX=(\tilX,\tilB)_\ssing$.
\end{lem}
\begin{proof}
We should prove that $(X,B)$ is regular at a point $x\in\tilX$ if and only if $(\tilX,\tilB)$ is regular at $x$. The direct implication is clear. Conversely, assume that $(\tilX,\tilB)$ is regular at $x$. Localizing we can assume that $X$ is local with closed point $x$. Since $\tilX$ is a divisor in $X$ and $\tilX$ is regular by Lemma~\ref{snclem}, $X$ is also regular and hence catenary. Set $I=I(x)$ and $\tilI=I\setminus\{l\}$ and observe that $B_I=\tilB_\tilI$ is regular of codimension $|\tilI|$ in $\tilX$. Since $\tilX$ is a divisor and $X$ is catenary, $B_I$ is of codimension $|\tilI|+1=|I|$, as required.
\end{proof}

\subsection{Blow up sequences and basic operations}\label{blsec}
In \S\ref{blsec} we will study transforms of the boundaries under blow up sequences. One easily sees that the strict transform of a locally principal subscheme does not have to be locally principal, so the strict transform is useless in this context. Although for a boundary $B=\{B_i\}_{i\in I}$ on $X$ and a blow up $f\:\Bl_V(X)\to X$ the full transform $f^*(B)=\{f^*(B_i)\}_{i\in I}$ is defined, it is not the transform one usually uses. Many components of $f^*(B)$ may contain the exceptional divisor $E=f^*(V)$ and usually one tries to split off redundant copies of $E$, at least to some extent. This leads to definitions of principal and complete transforms given below. Also, we will view the exceptional component $E_f=\{E\}$ as a boundary rather than a single divisor. This becomes sensitive when extending the notions of the transforms to blow up sequences $f$, since the new boundary $E_f=\{E_1\.E_n\}$ keeps track of the order of blow ups of $f$.

\subsubsection{Principal transform of closed subschemes}

\begin{rem}
An important role in embedded resolution of singularities is played by a so called principal (weak or controllable) transform of ideals or marked ideals under blow ups. It is obtained from the full (or total) transform $f^*(Z)$ by removing an appropriate multiple (depending on the setting) of the exceptional divisor. Thus, the principal transform is a small step from the full transform towards the strict one, which still can be easily described by explicit formulas in terms of the corresponding ideals.
\end{rem}

Let $f\:X'\to X$ be the blow up along $V\into X$ and let $Z\into X$ be a closed subscheme. For our needs it will be convenient to adopt the following variant of principal transform of $Z$ under $f$. Decompose $V$ as $T\coprod W$, where $T$ is the union of all connected components of $V$ that are closed subschemes in $Z$. Then the exceptional divisor $V'=f^*(V)$ possesses a component $T'=f^*(T)$ which is contained in $f^*(Z)$, and hence the closed subscheme $f^*(Z)-T'$ is defined. We call it the principal transform of $Z$ and denote it by $f^\rhd(Z)$. The {\em principal transform} $g^\rhd(Z)$ with respect to a blow up sequence $g\:X_n\longto X_0=X$ is defined iteratively.

\begin{rem}\label{transformrem}
(i) In sharp contrast with strict and full transforms, the principal transform is not local on the base because it can happen that $V$ is connected and is not a subscheme of $Z$, in particular, $f^\rhd(Z)=f^*(Z)$, but $\emptyset\neq V|_U\into Z|_U$ for an open subscheme $U\into X$. The complete transform which will be defined later is also of a non-local nature.

(ii) Nevertheless, if $V$ is regular (or just has integral connected components) then the above problem cannot happen and all transforms are of local nature.
\end{rem}

\subsubsection{Principal transform of the boundary}
Note that the principal transform of any locally principal closed subscheme is a locally principal closed subscheme. Therefore, given a $B$-scheme $(X,B)$ with $B=\{B_i\}_{i\in I}$ and a blow up sequence $f\:X'\longto X$ we can define the {\em principal transform} $f^\rhd(B)$ as the tuple $\{f^\rhd(B_i)\}_{i\in I}$. By the very definition, the principal transform is compatible with compositions of blow up sequences.

\begin{rem}
An equivalent definition of the principal transform is given in \cite[4.4]{CJS}, where the transform is called ``principal strict transform". We prefer to change the terminology because there might be smaller principal transforms containing the strict transform (e.g. when we modify the definition as $f^\square(Z)=f^*(Z)-\sum_i n_if^*(V_i)$ where $V_i$'s are the connected components of the center of $f$ and $n_i$ is the maximal number for which $n_if^*(V_i)$ is a subscheme of $f^*(Z)$).
\end{rem}

\subsubsection{Complete transform of the boundary}
Assume that $(X,B)$ is as above and $X'=\Bl_V(X)$. Then we define the {\em complete transform} of the boundary $f^\circ(B)=f^\rhd(B)\sqord E_f$, where $E_f=\{E\}$ and $E=V\times_XX'$ is the {\em exceptional divisor} of the blow up along $V$. Note that $E$ depends on the blow up and is not determined only by the morphism $X'\to X$, and we adjoin $E$ as a new element even when $f^\rhd(B)$ contains its copies. We have that $|f^*(B)|\subseteq|f^\circ(B)|=|f^*(B)|\cup|E_f|$, because $[f^\circ(B)]=[f^\rhd(B)]+E=([f^*(B)]-nE)+E$ where $n$ is a number for which $[f^*(B)]-nE$ is defined. We warn the reader that unlike the principal transform, the complete transform is not additive, i.e. $f^\circ(B\coprod B')\neq f^\circ(B)\coprod f^\circ(B')$ even as unordered sets.

If $f\:X'\longto X$ is a general blow up sequence then we define the complete transform $f^\circ(B)$ iteratively. In particular, $f^\circ(B)$ is the ordered disjoint union of the {\em old boundary} $f^\rhd(B)$ and a {\em new boundary} $E_f$ which we also call the {\em boundary of the blow up sequence} $f$.

\begin{exam}\label{blexam}
(i) Assume that $V=B_i\in B$ is a Cartier divisor whose connected components are not contained in any $B_j$ with $j\neq i$. Then $X'\toisom X$ and $f^\circ(B)$ equals to $B$ as unordered sets. However, the order is different because we move $B_i$ to be the last element in the boundary. Indeed, $f^\rhd(B_i)=\emptyset$ and so we remove the $i$-th component, but $E_f=B_i$ and so we adjoin the same component with the largest index.

(ii) If $B_i=nV$ then the $i$-th boundary component disappears after $n$ blow ups with center at $V$.
\end{exam}

\subsubsection{Summary of transforms}\label{sumsec}
For the sake of referencing we collect basic properties of the transforms in the following lemma. Since the strict transform $f^!(B)$ is not defined (at least as a boundary), we will consider $f^!(|B|)$ and $f^!([B])$ instead.

\begin{lem}\label{sumlem}
Let $f\:X'\longto X$ be a blow up sequence with new boundary $E_f$ and let $B=\{B_1\. B_n\}$ be a boundary on $X$, then

(i) $|E_f|$ is the reduced exceptional divisor, i.e. $|E_f|$ is the smallest closed subset of $X'$ such that $f$ restricts to a composition of trivial blow ups over $X\setminus f(|E_f|)$.

(ii) The total and principal transforms are componentwise in the sense that $f^*(B)=\{f^*(B_1)\. f^*(B_n)\}$ and $f^\rhd(B)=\{f^\rhd(B_1)\. f^\rhd(B_n)\}$, and the complete transform is obtained from the principal transform by adjoining $E_f$, i.e. $f^\circ(B)=f^\rhd(E)\sqord E_f$.

(iii) We have componentwise inclusions of strict, principal and full transforms $f^!(B_i)\into f^\rhd(B_i)\into f^*(B_i)$, where the last two schemes are always locally principal. In addition, $f^*(B_i)=f^\rhd(B_i)+D_i$, where $D_i$ is an exceptional divisor (i.e. $|D_i|\subset|E_f|$).

(iv) On the level of supports the transforms are related as follows: $|f^!(|B|)|\subseteq|f^\rhd(B)|\subseteq|f^*(B)|\subseteq|f^\circ(B)|=|f^!(|B|)|\cup|E_f|$.

(v) On the level of divisorial supports the transforms are related as follows: $f^!([B])\into[f^\rhd(B)]\into[f^*(B)]$ and $[f^\rhd(B)]+[E_f]=[f^\circ(B)]$.

(vi) Principal transform of closed subschemes agrees with the principal transform of marked ideals of order one. More concretely, assume that $\ucalI=(N,\emptyset,N,\calI,1)$ is a marked ideal with a permissible blow up sequence $f\:N'\longto N$ and let $\ucalI'=(N',E',N',\calI',1)$ be the transform of $\ucalI$ (see \cite[\S2]{BMfun}). Then $Z'=f^\rhd(Z)$ where $Z'\into N'$ and $Z\into N$ are the closed subschemes defined by $\calI'$ and $\calI$. In particular, if $f$ is a resolution of $\ucalI$ then $f^\rhd(Z)=\emptyset$.
\end{lem}
Note that $f^!(B_i)$ denotes the strict transform of the boundary component $B_i$, which is a scheme. Also, it is easy to see that all embeddings do not have to be equalities. For example, if $B$ is empty and $f$ is non-trivial then $|f^*(B)|=\emptyset\neq E_f=|f^\circ(B)|$.
\begin{proof}
The assertions (i)--(vi) are easily verified by induction on the length, and many statements just repeat the definitions. We will only prove (vi) to illustrate this. Let $V$ be the first center of $f$. Then $V$ is in the cosupport of $\ucalI$ which is precisely $|Z|$, and hence $Z_1=f^\rhd(Z)=f^*(Z)-f^*(V)$. On the other hand, $\calI_1=\calI_E^{-d}(\calI\calO_{N'})$ where $E=f^*(V)$ is the exceptional divisor and $d=1$ is the order of $\ucalI$. This implies the claim for a sequence of length one, and we deduce that $\cosupp(\ucalI_1)=|Z_1|$ contains the second center of $f$. So, we can repeat the argument for the second blow up, etc., thus mastering induction on the length.
\end{proof}

\begin{rem}
Typically, one operates with principal and complete transforms when building a desingularization functor, and a desingularization is often achieved by getting an empty $f^\rhd(B)$ and an snc $f^\circ(B)$. It then follows from Lemma \ref{sumlem}(iv) that $f^!(|B|)$ is empty and $f^*(B)$ is strictly monomial.
\end{rem}

\subsubsection{$B$-blow up sequences}
We introduced $B$-schemes in order to control the boundaries. In particular, it is important to control the boundary of the blow up sequences. So, we define the {\em $B$-blow up} $F\:(X',B')\to(X,B)$ as a blow up $f\:X'\to X$ such that $B'=f^\circ(B)$. By the {\em center} of $F$ we mean the center of $f$ and say that $F$ is {\em trivial} if $f$ is trivial.

\begin{rem}\label{blowrem}
(i) For this paper it suffices to use the above formal definition of $B$-blow ups, but we note for the sake of completeness that a $B$-blow up possesses a natural structure of a morphism of $B$-schemes. Indeed, for each $B_i\in B$ we have that $f^*(B_i)=f^\rhd(B_i)+nE_f$, where $n\in\{0,1\}$ and both $f^\rhd(B_i)$ and $E_f$ are components of $B'=f^\circ(B)$. For the sake of comparison, $(X',f^\rhd(B))\to(X,B)$ usually cannot be equipped with a structure of a $B$-morphism.

(ii) Unlike blow ups of schemes, usually a $B$-blow up $F$ is not an isomorphism of $B$-schemes even when its center $V$ is a Cartier divisor. Actually, one easily sees that $F$ is an isomorphism if and only if either $f$ is trivial or we are in the situation of Example \ref{blexam}(i) and in addition $V=B_i$ is the largest element of the boundary. In the second case, the largest component of the boundary is killed by the transform and then adjoined again as the new component.
\end{rem}

Naturally, a {\em $B$-blow up sequence} $F\:(X',B')\longto (X,B)$ is a sequence of $B$-blow ups. Giving such a sequence with target $(X,B)$ is equivalent to giving a blow up sequence $f\:X'\longto X$ because the boundaries are uniquely determined as complete transforms. Therefore, given a $B$-scheme $(X,B)$ we will pass freely between blow up sequences of $X$ and $B$-blow up sequences of $(X,B)$.

\subsubsection{Restriction of transforms onto closed subschemes}
Assume that $f\:X'\to X$ is a blow up along $V$ and $Z\into X$ is a closed subscheme. An easy explicit computation on the charts of $\Bl_V(X)$ shows that the strict transform $Z'=f^!(Z)$ is the blow up of $Z$ along $\tilV=V|_Z$, see \cite[\S1]{Conrad}. This also implies that the restriction of the exceptional divisor $V\times_XX'$ onto $Z'$ is the exceptional divisor $\tilV\times_ZZ'$ of the blow up $\tilf\:Z'\to Z$. Assume now that $B$ is a boundary on $X$ and $\tilB=B|_Z$ is its restriction on $Z$. By the transitivity of fibred products we obviously have that $(B\times_XX')|_{Z'}=\tilB\times_ZZ'$. In other words, the full transform $f^*$ is compatible with the restriction onto closed subschemes.

The situation with other transforms is more delicate: the equalities $f^\rhd(B)|_{Z'}=\tilf^\rhd(\tilB)$ and $f^\circ(B)|_{Z'}=\tilf^\circ(\tilB)$ fail if and only if there exists $B_i\in B$ and a connected component $V'\into V$ such that $V'$ is not a subscheme of $B_i$ but $V'|_Z$ is a subscheme of $B_i|_Z$. Indeed, the condition on $B_i$ is satisfied if and only if the transforms of $B_i$ and $B_i|_Z$ under $f$ and $\tilf$ are computed using different cases. However, this bad situation cannot occur whenever $V\into Z$, so we at least have the following lemma.

\begin{lem}\label{reslem}
Let $X$ be a scheme with a closed subscheme $Z$ and assume that $f\:(X',B')\longto(X,B)$ is a $B$-blow up sequence whose centers are closed subschemes of the strict transforms of $Z$. Let $\tilB=B|_Z$ and let $\tilf\:Z'\longto Z$ denote the induced blow up sequence of strict transforms. Then $f^\rhd(B)|_{Z'}=\tilf^\rhd(\tilB)$ and $f^\circ(B)|_{Z'}=\tilf^\circ(\tilB)$. In particular, $(Z',B'|_{Z'})\longto (Z,B|_Z)$ is a $B$-blow up sequence which will be denoted $f|_Z$.
\end{lem}

\subsection{Permissible $B$-blow up sequences}

\subsubsection{Transversality to the boundary}
Given a $B$-scheme $(X,B)$ and a closed subscheme $Z\into X$ we say that $Z$ is {\em transversal to the boundary} if $(Z,B|_Z)$ is a regular $B$-scheme. A more traditional way to formulate this condition is to say that each scheme $Z\times_X B_{i_1}\times_X\dots B_{i_n}$ is either empty or regular of codimension $n$ in $Z$. In particular, taking $n=0$ we see that $Z$ itself is regular. In the important particular case when $(X,B)$ is regular this just means that $Z$ is regular and transversal to the snc divisor $|B|$ in the usual sense.

\subsubsection{Simple normal crossings with the boundary}
More generally, we say that $Z$ has {\em simple normal crossings with the boundary} if $(Z,B|_Z)$ is a semi-regular $B$-scheme. As earlier, in the case when $(X,B)$ is regular this reduces to the usual notion of being regular and having simple normal crossings with an snc divisor.

\subsubsection{Permissibility}
A $B$-blow up $(X',B')\to(X,B)$ is called {\em permissible} if its center $V\into X$ has simple normal crossings with the boundary $B$. More generally, a $B$-blow up sequence is {\em permissible} if all its $B$-blow ups are permissible.

\begin{rem}
(i) Even when $B$ is empty the permissibility condition is stronger than just blowing up along regular centers. Namely, the first center just has to be regular, but after that a non-empty boundary appears, and this imposes an additional restriction on the choice of further centers.

(ii) For the sake of giving an inductive proof, even when interested in embedded desingularization without boundaries one often has to treat the exceptional divisors with a certain respect, and practically this amounts to considering only permissible blow up sequences. In particular, the known embedded desingularization algorithms for varieties construct permissible blow up sequences.
\end{rem}

\begin{lem}
If $(X,B)$ is a regular $B$-scheme with the boundary $B=\{B_1\.B_n\}$ and $(X',B')\longto(X,B)$ is a permissible $B$-blow up sequence then $(X',B')$ is regular and $f^\rhd(B)=\{f^!(B_1)\.f^!(B_n)\}$.
\end{lem}
\begin{proof}
We can assume that $f$ is a single blow up along $V$. Let $B_i$ be a component of $B$ and $V_0$ a connected component of $V$. It is easy to see that both in the case when $V_0$ is transversal to $B_i$ and in the case when $V_0$ is contained in $B_i$ one has that $f^\rhd(B_i)=f^!(B_i)$ in a neighborhood of $f^*(V_0)$. So, $f^\circ(B)=\{f^!(B_1)\.f^!(B_n),E_f\}$ and it is well known that the latter is an snc divisor whenever $V$ has simple normal crossings with $B$.
\end{proof}

\subsubsection{$B$-permissibility}
Often it will be convenient to express the permissibility in terms of usual blow up sequences and the initial boundary $B$. So, given a $B$-scheme $(X,B)$ we say that a blow up sequence $X'\longto X$ is {\em $B$-permissible} if the induced $B$-blow up sequence $(X',B')\longto(X,B)$ is permissible.

\subsubsection{Pushforward and restriction}
We refer to \cite[\S4.2.1]{nemb} for the definition of the pushforward of a blow up sequence with respect to a closed immersion.

\begin{lem}\label{permisslem}
Assume that $X$ is a scheme with a boundary $B$ and closed subscheme $i\:\tilX\into X$, and set $\tilB=B|_\tilX$.

(i) If $\tilf\:\tilX'\longto\tilX$ is a $\tilB$-permissible blow up sequence then the pushforward $f=i_*(\tilf)$ is a $B$-permissible blow up sequence.

(ii) If $f\:X'\longto X$ is a $B$-permissible blow up sequence whose centers are contained in the strict transforms of $\tilX$ then the induced blow up sequence of strict transforms $\tilf\:f^!(\tilX)\longto\tilX$ (see \cite[\S2.2.7]{nemb}) is $\tilB$-permissible.
\end{lem}
\begin{proof}
Note that both in (i) and (ii) the centers of $f$ lie on the strict transforms of $\tilX$, and hence the transforms of the boundaries in the blow up sequences $f$ and $\tilf$ are compatible with the restriction by Lemma \ref{reslem}. Now, an obvious induction on the length of $f$ reduces the proof of both (i) and (ii) to the claim that if both $\tilf$ and $f$ are a single blow up along $V\into\tilX$ then either $f$ is $B$-permissible and $\tilf$ is $\tilB$-permissible or both blow ups are not permissible. But the latter is clear since $V\into\tilX$ and hence $B_i|_V=\tilB_i|_V$.
\end{proof}

\subsection{Desingularization of $B$-schemes}

\subsubsection{Desingularization of a $B$-scheme}
By {\em desingularization} of a $B$-scheme $(X,B)$ we mean a $B$-blow up sequence $f\:(X',B')\longto(X,B)$ such that the $B$-scheme $(X',B')$ is semi-regular and the centers of $f$ are disjoint from the preimages of $(X,B)_\sreg$. If, in addition, the centers of $f$ are regular (resp. $B$-permissible) then we say that the desingularization is {\em strong} (resp. {\em $B$-strong}).

\begin{rem}\label{desrem}
(i) A desingularization $f\:(X',B')\longto(X,B)$ can be extended in an obvious way to a $B$-blow up sequence that modifies the non-regular locus of $(X,B)$ and produces a regular $B$-scheme $(X'',B'')$: one just has to kill the components of $X'$ contained in $|B|$ by blowing them up.

(ii) Also, it is easy to produce from $f$ a $B$-blow up sequence $g\:(X'',B'')\to(X,B)$ such that $(X'',B'')$ is semi-regular, $g^\rhd(B)=\emptyset$ and the centers of $g$ sit over $|B|\cup X_\sing$. This is done by blowing up the strata of $B'$ starting with the smallest ones. Namely, we blow up $B'_I$ at the first stage. This resolves the strict transform of $\cup_{i\in I}B'_{I\setminus\{i\}}$, so we can blow it up at the second stage and proceed similarly until the strict transform of $\cup_{i\in I}B'_i$ is blown up at the last stage. The new sequence $g$, which is obtained by extending $f$ in this way, is as required.
\end{rem}

\subsubsection{Functoriality}
Let $\gtC$ be a category whose objects are certain $B$-schemes and whose morphisms are certain strict regular morphisms. Then a {\em functorial desingularization} on $\gtC$ is a rule $\calF$ which associates a desingularization $\calF(X,B)\:(X',B')\longto(X,B)$ to each $B$-scheme $(X,B)$ from $\gtC$ in a way {\em compatible} with the morphisms of $\gtC$. The latter means that for each $h\:(\oX,\oB)\to(X,B)$ in $\gtC$ the $B$-blow up sequence $\calF(\oX,\oB)$ is obtained from $h^*(\calF(X,B))$ by omitting all trivial $B$-blow ups.

\subsubsection{The case of varieties}
Let $\calF_\Var$ be the non-embedded desingularization functor from \cite[Theorem~6.1]{BMT}. Recall that it associates to a variety $X$ of characteristic zero a strong desingularization $\calF(X)\:X'\longto X$ and is compatible with all regular morphisms. Moreover, the addendum to \cite[Theorem~6.1]{BMT} asserts that the associated $B$-blow up sequence $(X',B')\longto(X,\emptyset)$ is a $B$-strong desingularization. In other words this can be formulated as follows.

\begin{theor}\label{bvardesth}
There exists functorial $B$-strong desingularization $\calF_\Var$ on the category $\BVar_{B=\emptyset}$ whose objects are finite disjoint unions of $B$-varieties of characteristic zero with empty boundary and whose morphisms are all regular morphisms.
\end{theor}

\subsection{Formal and analytic analogs}
All definitions concerning boundaries, $B$-schemes, $B$-blow ups, desingularization of $B$-schemes and functoriality apply almost verbatim to the contexts of qe formal schemes and complex or non-archimedean analytic spaces. Note that semi-regularity is preserved by the following functors: (1) formal completion of qe $S$-schemes along a fixed ideal on a base scheme $S$, (2) analytification of $k$-varieties where $k$ is a complete field (either archimedean or non-archimedean). In particular, it follows that the completion and analytification functors take desingularizations of $B$-schemes to desingularizations of formal $B$-schemes or analytic $B$-spaces.

\section{Desingularization of qe $B$-schemes of characteristic zero}\label{dessec}

\subsection{$B$-Varieties}\label{bvarsec}
Let $\BVar$ be the category of finite disjoint unions of $B$-varieties of characteristic zero with all regular strict morphisms between them.

\begin{rem}\label{BVarrem}
One can extend the argument from \cite[Addendum~6.1]{BMT} to the case of a non-empty initial boundary obtaining a $B$-strong desingularization on $\BVar$. However, this requires a serious adjustment and will be done elsewhere. In this paper we will only use $\calF_\Var$ as a black box to construct a strong desingularization on $\BVar$, and this desingularization is not $B$-strong.
\end{rem}

\begin{theor}\label{bvarth}
There exists a strong desingularization functor $\calF_\BVar$ on the category $\BVar$.
\end{theor}
\begin{proof}
We will construct a strong desingularization $\calF_\BVar(X,B)$ of an object $(X,B)$ from $\BVar$ in three steps $\calF_i(X_i,B_i)\:(X_{i+1},B_{i+1})\longto(X_i,B_i)$, where $(X,B)=(X_0,B_0)$ and $0\le i\le 2$. In fact, $\calF_0$ resolves $X$, $\calF_1$ resolves an ``outer" part of the boundary, and $\calF_2$ completes the job. Functoriality will be clear from the construction.

Let $B=\{H_1\. H_n\}$ and $Z=(X,B)_\ssing$. We will only use $B$-blow up sequence whose centers are regular and {\em $Z$-supported}, i.e. contained in the preimage of $Z$. Such sequences will be called admissible throughout the proof. Also, we will use induction on $d=\dim(X)$. If $d=0$ then we define $\calF_\BVar(X,B)$ to be the blow up along $Z$. So, assume that $d>0$.

Step 0. {\it There exists a functorial admissible $B$-blow up sequence $\calF_0\:(X_1,B_1)\longto(X,B)$ such that $X_1$ is regular.} The strong non-embedded desingularization functor from \cite[Theorem 6.1]{BMT} applied to $X$ produces a strong desingularization $X_1\longto X$. Let $\calF_0(X,B)$ be the corresponding $B$-blow up sequence.

Assumption 1. {\it We can assume in the sequel that $X$ is regular and $B$ is divisorial.} After step 0, we should resolve $(X_1,B_1)$ with a regular $X_1$. To simplify notation, we replace $(X,B)$ by $(X_1,B_1)$ and assume in the sequel that $X$ is regular. In addition, by the regularity of $X$, each component of $B$ decomposes as $H_i=H^\div_i\coprod H^\comp_i$ where $H^\div_i$ is a Cartier divisor and $H^\comp_i$ is a component of $X$ in the sense of \S\ref{convsec}. Let $B^\div$ denote the boundary $\{H^\div_1\.H^\div_n\}$. Clearly, any strong desingularization $X'\longto X$ of $(X,B^\div)$ is also a strong desingularization of $(X,B)$ and hence we can replace $B$ by $B^\div$, making it divisorial.

Since $X$ is regular, each $H_i$ splits as $H^\out_i+H^\inn_i$, where $H^\inn_i$ is a $Z$-supported divisor and any irreducible component of the divisor $H^\out_i$ is not $Z$-supported. We call $B^\out=\{H^\out_1\.H^\out_n\}$ and $B^\inn=\{H^\inn_1\.H^\inn_{n}\}$ the outer and inner parts of $B$. Dealing with $B^\out$ will require a special care. Given an admissible $B$-blow up sequence $f\:(X',B')\longto(X,B)$ we will use the notation $B'=\{H'_1\.H'_{n'}\}$. In particular, $H'_i:=f^\rhd(H_i)$ for $1\le i\le n$ and $B'=\{H'_1\. H'_n\}\sqord E_f$, where $E_f$ is the boundary of $f$. Since $X$ and the centers of $f$ are regular, $X'$ is regular and we can split $B'$ as the sum of $Z$-supported and completely not $Z$-supported parts $B'^\inn$ and $B'^\out$. Note that $H'^\out_i=f^!(H^\out_i)$ for $1\le i\le n$ and $H'^\out_i=\emptyset$ for $n<i\le n'$.

Step 1. {\it There exists a functorial admissible $B$-blow up sequence $\calF_1\:(X_2,B_2)\longto(X,B)$ such that $B_2^\out$ is snc.} We will compose $n$ sequences. For a number $l$ with $1\le l\le n$ define a $B$-blow up sequence $g_l(X,B)$ as follows. Set $\tilX=H^\out_l$, $C=\{H^\out_{l+1}\.H^\out_n\}$ and $\tilC=C|_\tilX$. By induction on the dimension, there exists a strong desingularization $\tilf\:(\tilX',\tilC')\longto(\tilX,\tilC)$. Let $f\:X'\longto X$ be the pushforward of $\tilX'\longto\tilX$, and let $g_l\:(X',B')\longto(X,B)$ and $(X',C')\longto(X,C)$ be the corresponding $B$-blow up sequences.

Note that $g_l$ is $\tilZ$-supported, where $\tilZ=(\tilX,\tilC)_\sing$, and $\tilZ\subseteq(X,C)_\sing$ by Lemma~\ref{locuslem}. In particular, $g_l$ is admissible. Note also that $H'^\out_l=f^!(H^\out_l)=\tilX'$ and $H'^\out_i\into f^\rhd(H^\out_l)=C'_i$ for $l<i\le n$. Since $C'|_{\tilX'}=\tilC'$ by Lemma~\ref{reslem}, we obtain from Lemma~\ref{locuslem} that the boundary $\{H'^\out_l,C'_2\.C'_n\}$ is snc in a neighborhood $U$ of $H'^\out_l$. Therefore the smaller boundary $\{H'^\out_l\.H'^\out_n\}$ is also snc in $U$.

We define $\calF_1$ as the composition of the sequences $g_1(X,B)\:(X',B')\longto(X,B)$, $g_2(X',B')\:(X'',B'')\longto(X',B')$, etc. By the properties of $g_l$ we then have that $B'^\out$ is snc in a neighborhood $U$ of $H'^\out_1$. Furthermore, $g_2$ only modifies the non-snc locus of $\{H'^\out_2\.H'^\out_n\}$ and hence is an isomorphism over $U$. Therefore, $B''^\out$ is snc in a neighborhood of $H''^\out_1\cup H''^\out_2$, etc. Composing $g_1\.g_n$ we resolve the whole outer boundary.

Assumption 2. {\it We can assume in addition that $B^\out$ is snc.} Indeed, after Step 1 it we should resolve $(X_2,B_2)$ with an snc $B_2^\out$. So we replace $(X,B)$ with $(X_2,B_2)$, for shortness.

Step 2. {\it Resolution of $B$.} Note that the divisor $D=[B^\inn]$ is $Z$-supported. Since $B^\out$ is snc, we can consider the marked ideal $\ucalI=(X,B^\out,X,\calI_{D},1)$, where marked ideals are as defined in \cite[\S5]{BMT}. Let $f\:X'\longto X$ be the resolution of $\ucalI$. We claim that the associated $B$-blow up sequence $\calF_2(X,B)\:(X',B')\to(X,B)$ is a strong desingularization of $(X,B)$.

First, $f$ depends functorially on $(X,B)$ and is $D$-supported and hence $Z$-supported. Since $X$ is regular and $f$ blows up regular centers, $X'$ is regular. It remains to show that $f^\circ(B)$ is snc. Note that $f^\rhd(D)=\emptyset$ by Lemma \ref{sumlem}(vi). In addition, if $E_f$ denotes the boundary of $f$ then $f^\rhd(B^\out)\sqord E_f$ is snc because $f$ is $B^\out$-permissible. Since $H^\inn_i\into D$ for $1\le i\le n$, we have that $f^\rhd(H^\inn_i)\into f^\rhd(D)=\emptyset$, and hence $f^\rhd(H_i)=g^\rhd(H^\out_i)$. Therefore, $f^\circ(B)=f^\rhd(B^\out)\sqord E_f$ is snc.
\end{proof}

\subsection{Formal $B$-varieties with small singular locus}\label{formvarsec}

\subsubsection{Locally principal formal $B$-schemes}
By a {\em locally principal formal $B$-scheme} we mean a triple $(\gtX,\gtB,\gtI)$ where $(\gtX,\gtB)$ is a formal $B$-scheme and $\gtI$ is an invertible ideal of definition of $\gtX$. Note that $\gtZ=\Spf(\calO_\gtX/\gtI)$ is a closed formal subscheme with reduction $\gtX_s$. Sometimes we will replace $\gtI$ by $\gtZ$ in the triple. By a morphism $(\gtX',\gtB',\gtZ')\to(\gtX,\gtB,\gtZ)$ of such creatures we always mean a morphism of formal $B$-schemes $f\:(\gtX',\gtB')\to(\gtX,\gtB)$ such that $\gtZ'=f^*(\gtZ)$. Actually, we will only be interested in the cases when $f$ is either a strict regular morphism or a formal $B$-blow up sequence.

\begin{rem}
(i) Both $\gtZ$ and the components of $\gtB$ are locally principal closed formal subschemes. However, they transform differently under formal $B$-blow ups. This is the reason to distinguish $\gtZ$ rather than to include it as a special (e.g. minimal) component of the boundary.

(ii) The role of $\gtZ$ will be to control the bad locus of the formal $B$-scheme $(\gtX,\gtB)$.
\end{rem}

\subsubsection{Rig-regularity}\label{rigsec}
Recall that a qe formal scheme $\gtX$ is called {\em rig-regular} if its singular support $\gtX_\sing$ is given by an open ideal. In particular, a desingularization of such an $\gtX$ involves only blow ups along open ideals. Intuitively, rig-regularity means that the generic fiber of $\gtX$ is regular, and this makes perfect sense in cases it is defined. For example, when $\gtX$ is of finite type over a complete DVR $R$ and the generic fiber of $\gtX$ is defined as a rigid space over $\Frac(R)$ (hence the terminology).

\subsubsection{The category $\hatBVar_\smal$}\label{formcatsec}
We now introduce the category $\hatBVar_\smal$ whose objects are finite disjoint unions of certain locally principal formal varieties with small boundaries. More concretely, $(\gtX,\gtB,\gtI)$ is in $\hatBVar_\smal$ if $(\gtX,\gtB)$ is a finite disjoint union of formal $B$-varieties of characteristic zero, $\gtX$ is rig-regular and $\gtB$ is $\gtZ$-supported, i.e. all components of $\gtB$ are supported on the closed fiber $\gtX_s$. A morphism of $\hatBVar_\smal$ is a strict regular morphism between its objects.

\begin{rem}
(i) When $\gtB$ is empty we obtain the category $\hatVar_{p=0}$ from \cite[\S3.1]{nemb}.

(ii) Assume that $(\gtX,\gtB,\gtI)$ is an object of $\hatBVar_\smal$ and $(X,\calI)$ is an algebraization of $(\gtX,\gtI)$ in the sense of \cite[\S3.1]{nemb}. Since $\gtB$ is $\gtI$-supported, it algebraizes uniquely to an $\calI$-supported boundary $B$ on $X$ and $B_i=\gtB_i$ as schemes. In particular, this algebraization uniquely extends to an algebraization $(X,B,\calI)$ of the original triple. Moreover, all components of $B$ are closed subschemes in the $n$-th fibers $X_n=\Spec(\calO_X/\calI^n)$ for a large enough $n$. It follows that all results from \cite[\S3.2]{nemb} obviously generalize to the objects of $\hatBVar_\smal$. For example, an analog of \cite[Proposition~3.2.1]{nemb} holds true and the proof is almost identical with the only minor modification that one should always consider thick enough $n$-th fibers so that they contain $B$.
\end{rem}

\subsubsection{Desingularization on $\hatBVar_\smal$}\label{formalsec}
Now we are ready to generalize \cite[Theorem~3.1.5]{nemb} to formal $B$-schemes.

\begin{theor}\label{formdesth}
Let $\calF_\BVar$ be a desingularization functor on $\BVar$. Then there exists unique up to unique isomorphism desingularization functor $\hatcalF_\BVar$ on $\hatBVar_\smal$ such that $\hatcalF_\BVar$ is compatible with $\calF_\BVar$ under formal completions. Moreover, $\hatcalF_\BVar$ is strong (resp. $B$-strong) if and only if $\calF_\BVar$ is strong (resp. $B$-strong).
\end{theor}
\begin{proof}
The argument repeats the proof of \cite[Theorem~3.1.5]{nemb}, as given in \cite[\S\S3.2-3.3]{nemb}, with the only minor modification that one should always consider thick enough $n$-th fibers so that they contain the boundary.
\end{proof}

\begin{rem}\label{gtzrem}
The constructed algorithm depends on $\gtZ$, in particular, we establish its functoriality only with respect to morphisms that respect $\gtZ$. It seems probable that known algorithms for varieties can be extended to formal varieties. If this is the case then the dependency on $\gtZ$ is only an artefact of our method that uses Elkik's algebraization theorem.
\end{rem}

\subsection{$B$-schemes with small singular locus} \label{smallsec}
Our next aim is to generalize \cite[Theorem~3.4.1]{nemb} to $B$-schemes with small singular locus. Consider the category $\Bsch_\smal$ as follows. Objects of $\Bsch_\smal$ are triples $(X,B,Z)$, where $(X,B)$ is a noetherian qe $B$-scheme of characteristic zero and $Z\into X$ is a Cartier divisor which is a disjoint union of varieties and contains $(X,B)_\ssing$. In particular, $X$ is generically reduced. Morphisms $(X',B',Z')\to(X,B,Z)$ in $\Bsch_\smal$ are strict regular morphisms, i.e. regular morphisms $f\:X'\to X$ such that $B'=f^*(B)$ and $Z'=f^*(Z)$.

\begin{theor}\label{locth}
There exists a strong desingularization functor $\calF_\smal$ on $\Bsch_\smal$ which assigns to a triple $(X,B,Z)$ a desingularization of $(X,B)$ and is compatible with all morphisms from $\Bsch_\smal$.
\end{theor}
\begin{proof}
The functor $\hatcalF_\BVar$ from Theorem~\ref{formdesth} will be the main ingredient of our construction. In fact, any other strong desingularization functor on $\hatBVar_\smal$ would do the job too. Let $T=(X,B)_\ssing$ denote the bad locus. Also, when possible we decompose the boundary $B=\{H_1\.H_n\}$ as a sum of inner and outer boundaries $B^\inn$ and $B^\out$, where $H_i=H_i^\inn+H_i^\out$ for each $H_i\in B$, and the irreducible components of $|H_i^\inn|$ are exactly the $Z$-supported irreducible components of $H_i$. If exists, this decomposition is unique, and it always exists when $X$ is regular (or just locally factorial).

Case 1. {\it Empty outer boundary.} Let $\Bsch_0$ denote the full subcategory of $\Bsch_\smal$ whose objects have empty outer boundary (i.e. $|B|\subseteq Z$). We claim that $\hatcalF_\BVar$ induces a desingularization functor on $\Bsch_0$. Indeed, the formal completion of $(X,B,Z)$ along $Z$ is an object of $\hatBVar_\smal$, and hence it is resolved by the functorial $\hatZ$-supported $B$-blow up sequence $\hatcalF_\BVar(\hatX,\hatB,\hatZ)$. Similarly to the proof of \cite[Theorem~3.4.1]{nemb} this sequence algebraizes to a functorial $Z$-supported $B$-blow up sequence $\calF_0(X,B,Z)\:(X',B')\longto(X,B)$ which provides a strong desingularization of $(X,B)$. In particular, compatibility with all regular morphisms follows from \cite[Corollary~2.4.5]{nemb}.

Case 2. {\it Regular outer boundary.} Let $\Bsch_1$ denote the full subcategory of $\Bsch_\smal$ formed by the triples $(X,B,Z)$ for which the decomposition $B=B^\inn+B^\out$ is defined and $|B^\out|$ is disjoint from $T$. We claim that the functor $\calF_0$ from Case 1 trivially extends to a desingularization functor $\calF_1$ on $\Bsch_1$. Indeed, if $(X,B,Z)$ is an object of $\Bsch_1$ then $(X,B^\inn)_\ssing\subseteq T$ is disjoint from $|B^\out|$ and it follows that $\calF_0(X,B^\inn,Z)$ resolves $(X,B)$: it obviously resolves $(X,B)$ over $X\setminus|B^\out|$ and it does not change anything near $|B^\out|$. Thus, we simply set $\calF_1(X,B,Z)=\calF_0(X,B^\inn,Z)$.

Case 3. {\it The general case.} Now, we are going to construct desingularization $\calF_\smal(X,B,Z)$ of a general object of $\Bsch_\smal$. We will use induction on the dimension of $Z$, so assume that $\dim(Z)=d$ and the functor is already constructed for the smaller values of $d$. Our construction is similar to the construction of the functor $\calF_\BVar$ in the proof of Theorem \ref{bvarth}, though we will have to separate $B^\out$ from the bad locus of the whole $B$, and this will require more work. Again we will proceed in 3 steps: 0) reduce to the case when $X$ is regular and $B$ is divisorial, 1) separate $B^\out$ from the bad locus, 2) resolve $B$.

Step 0. {\it We can assume that $X$ is regular and each $H_i$ is a Cartier divisor. In particular, $(X,B)$ is regular outside of $Z$.} Note that $\calF_0(X,\emptyset,Z)$ induces a desingularization $f\:X'\longto X$ of $X$. (In fact, this is $\calF_\Var(X,Z)$ constructed in \cite[Theorem~3.1.5]{nemb}.) Set $B'=f^\circ(B)$ and $Z'=f^*(Z)$. Then $B'$ decomposes as $B^\div+B^\comp$, where $B^\div$ is divisorial and any component of $B^\comp$ is a component of $X$. Since any desingularization of $(X',B^\div)$ is also a desingularization of $(X',B')$ we can safely replace $(X,B,Z)$ by $(X',B^\div,Z')$ accomplishing the step. In particular, the decomposition $B=B^\inn+B^\out$ is now defined.

Now, we will define $\calF(X,B,Z)$ as a composition of several $T$-supported $B$-blow up sequences with regular centers. To ease the notation, we will consider the underlying blow up sequences of schemes and the intermediate sequences will be denoted $f\:X'\longto X$. In addition, we set $B'=f^\circ(B)$ and $Z'=f^*(Z)$. By the bad loci we mean the closed sets $T$ and $T'=(X',B')_\ssing$.

Step 1. {\it We can achieve that $(X',B',Z')$ is in $\Bsch_1$, thus separating the outer boundary from the bad locus.} We start with the following observation.

Claim (1). {\it It is enough for a component $H_l$ of $B$ to construct a blow up sequence with regular $T$-supported centers $g\:X'\longto X$ such that $H'^\out_l$ is disjoint from $T'$.} Indeed, given such a $g$ we can apply the same argument to $(X',g^\circ(B),g^*(Z))$ to find a $T'$-supported blow up sequence $X''\longto X$ which separates another outer boundary component from the bad locus, etc.

{\it General plan of constructing $g$.} We will only blow up regular $T$-supported centers lying on the strict transforms of $\tilX:=H_l^\out$. So, we will always have that $g$ is the pushforward of the blow up sequence $\tilg\:\tilX'\longto\tilX$, where $\tilX'=g^!(\tilX)$. Also, we will use the following notation: $\tilZ=Z|_\tilX$, $\tilC=C|_\tilX$ and $\tilC'=\tilg^\circ(\tilC)$, where $C=\{H_1\.H_{l-1},H_l^\inn,H_{l+1}\.H_n\}$.

Substep (a). {\it We can achieve that the $B$-scheme $(\tilX',\tilC')$ is regular.} Since $(X,B)$ is regular outside of $Z$, it follows that $(\tilX,\tilC)$ is regular outside of $\tilZ$. In particular, $(\tilX,\tilC,\tilZ)$ is an object of $\Bsch_\smal$. Furthermore, $Z$ does not contain irreducible components of $\tilX$, hence $\dim(\tilZ)=d-1$ and the desingularization $\calF_\smal(\tilX,\tilC,\tilZ)\:(\tilX',\tilC')\longto(\tilX,\tilC)$ is defined by the induction assumption. We define $g\:X'\longto X$ to be the pushforward of $\tilg\:\tilX'\longto\tilX$ with respect to the closed immersion $\tilX\into X$.

Substep (b). {\it We can achieve in addition that $\tilX'=H'^\out_l$ and $\tilX'$ is a component of $H'_l=g^\rhd(H_l)$.} This will be achieved by several blow ups  that do not modify $\tilX'$. In particular, the condition of substep (a) will be preserved. Since $X'$ is regular and $\tilX'$ is its regular subscheme of codimension one, $\tilX'$ is a divisor. Clearly $\tilX'\into H'_l$ and hence $H_l'=\tilX'+Y'$, where $Y'$ is a $Z$-supported divisor. Note also that the irreducible components of $\tilX'=g^!(H_l^\out)$ are not $Z$-supported. In particular, $\tilY':=Y'|_{\tilX'}$ is a divisor on $\tilX'$, and if it is empty then $H'_l$ splits as $\tilX'\coprod Y'$ and this splitting coincides with $H'^\out_l\coprod H'^\inn_l$.

It remains to achieve the situation when $\tilY'=\emptyset$. Note that $$|\tilY'|\subseteq |\tilX'|\cap\left(|E_g|\cup|f^\rhd(H_l^\inn\right)|)\subseteq|E_\tilg|\cup|\tilC'|=|\tilC'|$$ and $\tilC'$ is snc by substep (a). Let $D_1\.D_k$ be the components of $\tilC'$. Choose the minimal numbers $m_1\. m_k$ such that $\tilY'\into\sum_{i=1}^km_iD_i$. We will kill $\tilY'$ by several blow ups along $D_i$'s and their pullbacks.

Take the minimal $j$ with a non-zero $m_j$ and consider the blow ups $h\:X''=\Bl_{D_j}(X')\to X'$ and $X'''=\Bl_{E''}(X'')\to X'$, where $E''=h^*(D_j)$ is the exceptional divisor of $h$. Note that both blow ups have regular centers and are $D_j$-supported, and hence $Z$-supported. When restricted to $\tilX'$, both centers coincide with $D_j$, hence $\tilX'''\toisom\tilX'$. We claim that extending $g$ to the composition $X'''\longto X$ one achieves that the new $\tilY'$ is contained in $(\sum_{i=1}^km_iD_i)-D_j$. This follows from a straightforward computation with charts, that we postpone to Lemma \ref{chartlem} after the proof of the theorem. Repeating the same double blow up operation ($\sum_i m_i)$-times, we achieve that $\tilY'$ vanishes.

Claim (2). {\it The conditions from substeps (a) and (b) imply that $\tilX'$ is disjoint from the bad locus $T'$. In particular, this completes step 1.} Recall that $X'$ is a regular scheme, $\tilX'$ is a regular divisor by substep (a), and $g^\circ(C)|_{\tilX'}=\tilC'$ by Lemma~\ref{reslem}. Hence Lemma~\ref{locuslem} implies that $\{\tilX'\}\coprod g^\circ(C)$ is an snc boundary in a neighborhood $U$ of $\tilX'$. By substep (b), shrinking $U$ we can also achieve that $H'_l\cap U=\tilX'$. Let $F$ be obtained from $C$ by removing $H_l^\inn$. Then, ignoring the order, $g^\circ(B)=\{H'_l\}\coprod g^\circ(F)$. So, $g^\circ(B)|_U$ is contained in the snc boundary $\{\tilX'\}\coprod g^\circ(C)|_U$, and hence $g^\circ(B)|_U$ is itself snc. Thus, $T'$ lies in the complement of $U$ and is disjoint from $\tilX'=H'^\out_l$.

Step 2. {\it Resolution of $B$.} Let $\calF_1(X',B',Z')\:X''\longto X'$ be as defined in Case 2 above. The composition $g\:X''\longto X$ is a desired desingularization of $(X,B)$ that functorially depends only on the triple $(X,B,Z)$, so we set $\calF_\smal(X,B,Z)=g$.
\end{proof}

Finally, let us prove the lemma we used in substep (b) above.

\begin{lem}\label{chartlem}
Assume that $X$ is a regular scheme with a regular divisor $T$, and $Y$ is a divisor in $X$ such that $Y|_T\into \sum_{i=1}^n m_iD_i$, where $\{D_1\. D_n\}$ is an snc boundary on $T$. Let $f\:X''\to X'\to X$ be obtained by first blowing up $D_1$ and then blowing up $D'_1=D_1\times_XX'$. Then $T''=f^!(T)$ is isomorphic to $T$ and the restriction of $f^\rhd(T+Y)-T''$ onto $T''$ embeds into $m'_1D_1+\sum_{i=2}^n m_iD_i$, where $m'_1=\max(0,m_1-1)$.
\end{lem}
\begin{proof}
Since the centers are regular, the principal transform is compatible with localizations by Remark~\ref{transformrem}(ii). So, it is enough to check the claim \'etale-locally on $X$, and hence we can assume that $X=\bfA^{d+1}=\Spec(k[x,y])$, $T=V(x)$ and $D_i=V(x,y_i)$ for $y=(y_1\. y_d)$ and $1\le i\le d$ (it is harmless to add components $D_j$ that do not appear in $Y|_T$ and set $m_j=0$). Then $Y=V(\phi)$, where $\phi=y_1^{l_1}\dots y_d^{l_d}+xP(x,y)$ and $l_i\le m_i$. The blow up $g\:X'\to X$ along $D_1$ is covered by the $x$-chart and the $y_1$-chart. The first one is disjoint from $g^!(T)$ and hence is not relevant. The new coordinates in the $y_1$-chart are $(x'=\frac{x}{y_1},y_1\. y_d)$, hence the pullback of $\phi$ equals to $\phi'=y_1^{l_1}\dots y_d^{l_d}+y_1x'P(y_1x',y)$. Note also that $g^*(T+Y)=V(y_1x'\phi')$ and hence $Z:=g^\rhd(T+Y)$ coincides with $V(x'\phi')$.

The second blow up $h\:X''\to X'$ of $f$ is along $E_g$, hence $h$ is an isomorphism and the transform is computed as follows: if $l_1=0$ then $D'_1=V(y_1)\nsubseteq Z$ and so $h^\rhd(Z)=h^*(Z)=Z$; if $l_1>0$ then $D'_1\subseteq Z$ and so $h^\rhd(Z)=Z-D'_1$. In the first case $f^\rhd(T+Y)=Z$ is defined by $x'\phi'$, and in the second case, $f^\rhd(T+Y)=T-D_1$ is defined by the element $x'\phi'/y_1=x'y_1^{l_1-1}y_2^{l_2}\dots y_d^{l_d}+x'^2P(y_1x',y)$. Since $T''=V(x')$, the lemma follows.
\end{proof}

\begin{rem}
If $\hatcalF_\BVar$ is independent of $\gtZ$ (see Remark~\ref{gtzrem}) then $\calF_\smal(X,B,Z)$ is independent of $Z$ too.
\end{rem}

\begin{rem}\label{formrem}
Using Trushin's strengthening of Elkik's theorem, see \cite[Theorem~51]{Trushin} one can also construct a $B$-strong desingularization in Theorem~\ref{locth}. This will not be used and the argument is rather complicated, so we only mention the main line. The naive idea is to pass to a formal completion along $Z$ and algebraize the obtained rig-regular formal variety $\gtX$ with its rig-snc boundary $\gtB$. The bottleneck here is that Trushin's theorem only allows to algebraize $\gtX$ with a single rig-regular closed formal subscheme $\gtT$ rather than the entire $\gtB$. Naturally, one should take $\gtT$ to be the non-empty $\gtB(n)$ with the maximal possible $n$. Then, using algebraization, one can separate $\gtB(n)$ from the bad locus of $(\gtX,\gtB)$, and proceed by induction on $n$.
\end{rem}

\subsection{General $B$-schemes}\label{genbsec}

\subsubsection{Unresolved locus}
Similarly to \cite[\S4.1.1]{nemb}, when working on strong (resp. $B$-strong) desingularization of $B$-schemes by the {\em unresolved locus} $f_\sing$ of a $B$-blow up sequence $f\:(X',B')\longto(X,B)$ we mean the smallest closed subset $T\subseteq X$ such that $f$ is a strong (resp. $B$-strong) desingularization over $X\setminus T$. Also, we say that $f$ is a {\em desingularization up to codimension $d$} if $f_\sing\subseteq X^{>d}$. Recall that $X^{\le d}$ denotes that set of points of $X$ of codimension at most $d$, and $X^{>d}=X\setminus X^{\le d}$.

\subsubsection{Equicodimensional blow up sequences and filtration by codimension}
A $B$-blow up sequence $(X',B')\longto(X,B)$ is {\em equicodimensional} if the blow up sequence $X'\longto X$ is equicodimensional in the sense of \cite[\S4.1.3]{nemb}. For completeness, we note that there is a straightforward generalization of \cite[Lemma~4.1.3]{nemb} to desingularization of $B$-schemes, which we leave to the reader. We will only use the obvious (and weaker) observation that if $\{\calF^{\le d}\}_{d\in\bfN}$ is a compatible family of functorial equicodimensional desingularizations up to codimension $d$ then this family possesses a limit $\calF$, which is a desingularization functor.

\subsubsection{Construction of $\calF$}
Let $\Bsch$ denote the category of all qe noetherian $B$-schemes of characteristic zero with all strict regular morphisms, and let $\Bsch_\red$ be the full subcategory of generically reduced $B$-schemes.

\begin{theor}\label{calbth}
There exists a strong desingularization functor $\calF_\red$ on $\Bsch_\red$.
\end{theor}
\begin{rem}
The functor $\calF_\smal$ from Theorem~\ref{locth} will be the main ingredient in the proof, and any other strong desingularization functor on $\Bsch_\smal$ can be used instead. Moreover, we will show that if one starts with a $B$-strong desingularization functor on $\Bsch_\smal$ then one can construct a $B$-strong desingularization functor $\calF_\red$. This only requires minor modification in the construction that will be explicitly mentioned.
\end{rem}
\begin{proof}
For shortness, we will write $\calF$ instead of $\calF_\red$. Let $(X,B)$ be a $B$-scheme from $\Bsch_\red$. We will construct a desingularization $\calF(X,B)$ and it will be clear that all stages of the construction are functorial. Actually, we will build a compatible sequence of functors $\calF^{\le d}$ which provide an equicodimensional desingularization up to codimension $d$ and such that if a center $C_i\subset X_i$ of $\calF^{\le d}(X,B)$ is such that its image in $X$ is of codimension $d$ then $C_i$ sits over $T_{d-1}:=\calF^{\le d-1}(X,B)_\sing$. The construction will be done inductively, and we define $\calF^{\le 0}$ to be the empty $B$-blow up sequence. This works fine since a generically reduced $B$-scheme is semi-regular at all its maximal points.

In the sequel, we assume that the functors $\calF^{\le 0}\.\calF^{\le d-1}$ are already constructed, and our aim is to construct $\calF^{\le d}(X,B)$. The required $B$-blow up sequence will be obtained by modifying the $B$-blow up sequence $\calF^{\le d-1}(X,B)$. Let $n$ be the length of this sequence. To simplify notation and avoid double indexes, after each modification we will denote the obtained blow up sequence as $f\:(X_m,B_m)\longto(X_0,B_0)=(X,B)$. In particular, we start with $f=\calF^{\le d-1}(X,B)$ and $m=n$, and we will change $f$ and $m$ in the sequel. By our assumption, $T_{d-1}$ is a closed subset of $X^{\ge d}$, hence it has finitely many points of codimension $d$. Let $T$ denote the set of these points and let $\oT$ be the Zariski closure of $T$.

Step 1. {\it Construction of $\calF^{\le d}(X,B)$ in the particular case when $T$ is closed.} Note that in this case $\calF^{\le d-1}(X,B)$ is a desingularization over $X\setminus T$ and hence $\calF^{\le d}(X,B)$ will be a desingularization of the whole $X$. We will use the operation of inserting a blow up sequence into another blow up sequence, which is defined in \cite[Definition~4.2.2]{nemb}.

{\it Extension 0.} Provide $T$ with the reduced scheme structure and extend $f$ by inserting $\Bl_T(X)\to X$ as the first blow up. As an output we obtain a blow up sequence $\calF^{\le d}_0(X,B)\:X'_m\longto X'_0\to X_0=X$ of length $m+1$ where the first center (the inserted one) is regular. Set $f=\calF^{\le d}_0(X,B)$ and increase $m$ by one after this step. As an output we achieve that $T\times_XX'_i$ is a Cartier divisor in $X'_i$ for $i>0$.

{\it Extensions $1\. n$.} The last $n$ centers of $\calF^{\le d}_0(X,B)$ do not have to be suitable for a strong (resp. $B$-strong) desingularization. So, we will use $n$ successive extensions to make the centers regular (reps. $B$-permissible) in the case of strong (resp. $B$-strong) desingularization.

Let us describe the $i$-th extension with $1\le i\le n$. It obtains as an input a blow up sequence $f=\calF^{\le d}_{i-1}(X,B)$ in which only the last $n-i$ centers can be non-permissible (resp. non-regular) and outputs a blow up sequence $\calF^{\le d}_i(X,B)$ with only $n-i-1$ bad blow ups in the end. By our assumption, $(X_{i+1},B_{i+1})\to(X_i,B_i)$ is the first blow up of $f$ whose center $W$ can be non-permissible (resp. non-regular). The latter happens if and only if $T_W:=(W,B_i|_W)_\ssing$ (resp. $T_W:=W_\sing$) is not empty. Obviously, the bad locus $T_W$ sits over $T$ and hence $T_W\subseteq W\cap(X_i)^{\le d}\subseteq W^{\le d-1}$. In particular, $\calF^{\le d-1}(W,B_i|_W)$ is a strong (resp. $B$-strong) desingularization that we denote $h\:(W',\tilB')\longto(W,B_i|_W)$. Clearly, $h$ is $T_W$-supported and hence $T$-supported. By \cite[Lemma 4.2.1]{nemb}, the pushforward $H\:X'_i\longto X_i$ of $h$ with respect to the closed immersion $W\into X_i$ is a blow up sequence with the same centers. Moreover, $H$ is $B_i$-permissible in the $B$-strong case by Lemma \ref{permisslem}.

Let now $f'\:X'_m\longto X'_i\longto X_i\longto X_0$ be obtained from $f$ by inserting $H\:X'_i\longto X_i$ instead of $X_{i+1}\to X_i$. By \cite[Lemma 4.2.3]{nemb} the center of $X'_{i+1}\to X'_i$ is the strict transform of $W$, hence it is $W'$. Since $(W',\tilB')$ is semi-regular (resp. $W'$ is regular) and $\tilB'=B'_i|_{W'}$, only the last $i-1$ blow ups of $f'$ can have non-permissible (resp. non-regular) centers. So, we can set $\calF^{\le d}_i(X,B)=f'$ and replace the old $f$ with $f'$.

{\it Extension $n+1$.} At this stage we already have a blow up sequence $f=\calF^{\le d}_n(X,B)$ such that all its centers are as required. The last problem we have to resolve is that $(X_m,B_m)_\ssing$ does not have to be empty. However, we at least know that the bad locus is $T$-supported and hence is contained in the Cartier divisor $D=T\times_XX_m$, which is a disjoint union of varieties. So, the triple $(X_m,B_m,D)$ is an object of $\Bsch_\smal$ and hence $(X_m,B_m)$ possesses a strong (resp. $B$-strong) desingularization $\calF_\smal(X_m,B_m,D)$.

Step 2. {\it Construction of $\calF^{\le d}(X,B)$ in general.} Set $X_T=\coprod_{x\in T}\Spec(\calO_{X,x})$ and $B_T=p^*(B)$, where $p\:X_T\to X$ is the projection, and note that $f_T=\calF^{\le d}(X_T,B_T)$ was defined in Step 1. For each $x\in T$ let $f_x$ denote the restriction of $f_T$ onto $X_x=\Spec(\calO_{X,x})$ with all trivial blow ups inherited from $f_T$. By functoriality, $f_x$ is a trivial extension of $\calF^{\le d}(X_x,B|_{X_x})$.

Choose an open neighborhood $U\into X$ of $X^{\le d}$ such that the closures $\ox\in U$ of distinct points $x\in T$ are pairwise disjoint, and define $g_x\:U_x\longto U$ as the pushforward of $f_x$ with respect to the pro-open immersion $X_x\into U$. Since each $g_x$ is $\ox$-supported, \cite[Lemma 4.2.4]{nemb} implies that we can merge all $g_x$'s into a single blow up sequence $U'\longto U$. Let $f\:X'\longto X$ be the pushforward of $g$ with respect to the open immersion $U\into X$. It follows that $f$ is obtained from $\calF^{\le d-1}(X,B)$ by inserting several equicodimensional $\oT$-supported blow ups. In particular, $f$ is a trivial extension of $\calF^{\le d-1}(X,B)$ over $X\setminus\oT$, and $f$ coincides with $f_x$ over $X_x$ for each $x\in T$. Thus, $f$ desingularizes $(X,B)$ over $X^{\le d}$ and we set $\calF^{\le d}(X,B)=f$.

The functoriality of the construction is established by checking that all intermediate steps are functorial. This is straightforward and done in the same way as in the proof of \cite[1.2.1]{nemb}, so we omit the details.
\end{proof}

\begin{cor}\label{calbcor}
There exists a strong desingularization $\calF$ on $\Bsch$.
\end{cor}
\begin{proof}
We will use the strong desingularization functor $\calF_\red$ from Theorem~\ref{calbth}. Let $(X,B)$ be a $B$-scheme from $\Bsch$. Consider the reduction $\tilX$ of $X$ and let $X'\longto X$ be the pushforward of $\calF_\red(\tilX,\emptyset)\:\tilX'\longto\tilX$. This gives rise to a $B$-blow up sequence $(X',B')\longto(X,B)$ with regular centers. Since $\tilX'$ is the reduction of $X'$, we can further blow up all components of $\tilX'$ which underly generically non-reduced components of $X'$. This gives rise to a $B$-blow up $(X'',B'')\to(X',B')$ with a generically reduced source. Finally, if $(X''',B''')\longto(X'',B'')$ is the blow up sequence $\calF_\red(X'',B'')$, then the composition $(X''',B''')\longto(X,B)$ is a required strong desingularization $\calF(X,B)$ of $(X,B)$. Clearly, all three steps are functorial.
\end{proof}

\subsection{Proof of the main results}\label{proofsec}

\subsubsection{Theorems \ref{Bth}}
After unrolling the definitions, this theorem is nothing else but Corollary~\ref{calbth}.

\subsubsection{Theorem \ref{embth}}\label{proofsubsec}
To construct $\calF_\emb(X,E,Z)$ we will compose two $T$-supported blow up sequences with regular centers. We will denote the intermediate sequence $f\:X'\longto X$ and we set, in addition, $Z'=f^!(Z)$ and $E'=f^\circ(E)$.

Step 1. {\em We can achieve that $Z'$ is regular and has simple normal crossings with $E'$. } Consider the $B$-blow up sequence $\tilf\:(Z',B')\longto(Z,B)$, where $B=E|_Z$ and $\tilf=\calF(Z,B)$ with $\calF$ as in Theorem~\ref{Bth}, and let $f\:(X',E')\longto(X,E)$ be its pushforward. In particular, the centers of $f$ are as required, and $B'=E'|_{Z'}$ by Lemma \ref{reslem}. Thus, $E'_{Z'}$ is semi-regular and hence $Z'$ has simple normal crossings with $E'$.

Step 2. {\em Making the boundary snc.} The boundary $E'$ does not have to be snc since $\tilf$ and $f$ do not have to be permissible. However, $Z'$ has simple normal crossings with $E'$, hence the bad locus of $(X',E')$ is disjoint from $Z'$. Thus, $\calF(X',E')\:(X'',E'')\longto(X',E')$ does not modify a neighborhood of $Z'$. Note that $E''$ is snc because $(X'',E'')$ is semi-regular and $E''$ is divisorial. Finally, $\calF(X',E')$ is $T$-supported because $E'$ is snc at any point $x'\in X'$ such that $f$ is a local isomorphism at $x'$. So, the composition $g\:(X'',E'')\longto(X,E)$ satisfies all assertions of the theorem and we set $\calF_\emb(X,E,Z)=g$.

\subsubsection{Theorem \ref{compacth}}
First, let us fix an $S$-compactification $X$ of $Y$: by Nagata compactification theorem there exists a reduced proper $S$-scheme $X$ with a dense subscheme $S$-isomorphic to $Y$. We identify $Y$ with that subscheme and set $E=X\setminus Y$. Since $\calF_\princ(X,E)\:X'\to X$ does not modify $Y$ we can replace $X$ by $X'$ achieving that $X$ is regular and $E$ is an snc divisor. Next, consider the snc divisor $D=Y\setminus Y_0$ and let $D_1\.D_n$ be its irreducible components. By $\oD$ and $\oD_1\.\oD_n$ we denote their Zariski closure in $X$. It suffices to find a modification $f\:\oY\to X$ such that $f$ is an isomorphism over $Y$ and the preimage of $E\cup\oD$ is an snc divisor. We will construct $f$ as a composition of $n$ blow up sequences. At first step we apply $f_1=\calF_\emb(X,E,\oD_1)$, which only modifies $E$ because $D_1$ is regular. To simplify notation, replace $X,E$ and all $\oD_i$ with their preimages under $f_1$. In this way we achieve that $E_1:=E\cup\oD_1$ becomes snc. Next, apply $f_2=\calF(X,E_1,\oD_2)$. Since $D_1$ has simple normal crossings with $D_2$, the modification locus of $f_2$ is contained in $E$. Renaming $X,E,\oD_i$ as earlier we achieve that $E_2:=E\cup\oD_1\cup\oD_2$ is snc, etc. Obviously, after $n$ iterations we achieve that $X\setminus Y_0=E\cup\oD$ is snc. As we explained above, the blow ups do not modify $Y$, so we can set $\oY=X$.

\begin{rem}
There are two choices in our construction: the initial choice of the compactification, and the choice of the order of the components of $D$.
\end{rem}

\subsubsection{Other theorems}
We saw in the Introduction that Theorem \ref{divth} follows from Theorem \ref{Bth}, see Remark \ref{brem}(i). Also, we deduced Theorem \ref{princth} from Theorem \ref{embth}, and explained how Theorem \ref{catth} is proved. It, thus, remains to establish the semistable reduction theorem.

\section{Semistable reduction}\label{semisec}

\subsection{Reduction to $d$-regularization}\label{dregsec}

\subsubsection{Special morphisms}
We start with recalling some terminology on formal sche\-mes. We say that a morphism $\gtf\:\gtX\to\gtS$ of formal schemes is of {\em topologically finite type} if it is quasi-compact and locally of the form $$\Spf(A[[x_1\.x_m]]\{x_{n+1}\.x_n\}/I)\to\Spf(A).$$ Following Berkovich, we will also call such morphisms {\em special}. Note that a special morphism $\gtf$ is adic if and only if one can take $m=0$, and this happens if and only if $\gtf$ is of finite type.

\subsubsection{Distinguished formal schemes}
Let $R$ be a complete DVR with a uniformizer $\pi$ and residue field $k$ and let $\gtS=\Spf(R)$. For a formal $R$-scheme $\gtX$ we call $\gtZ=\gtX\otimes_Rk$ the {\em special fiber} of $\gtX$. It is a divisor containing the closed fiber $\gtX_s$.

Given $0\le m\le n$, let $R_{n,m}$ denote the completion of $R\{t_0,t_1\.t_n\}$ along the principal ideal $(t_0\dots t_m)$. We say that $\gtX$ is {\em strictly distinguished} (resp. {\em distinguished}) if locally (resp. \'etale-locally) on $\gtX$ one can factor the structure morphism into a composition $$\gtX\stackrel\phi\to\Spf\left(R_{n,m}/(t_0^{l_0}\dots t_n^{l_n}-a)\right)\to\gtS$$ such that $\phi$ is \'etale, $0\neq a\in \pi R$, and $l_i>0$ for $0\le i\le m$. If one can also achieve that $l_i\le 1$ for $0\le i\le n$ then $\gtX$ is called {\em strictly semistable} (resp. {\em semistable}).

\begin{rem}
(i) We require that $n\ge 0$ and $l_0\.l_m>0$ in order to ensure that $\gtX_s$ is the support of a divisor: $\gtX_s=V(t_0\dots t_m)$. In particular, $t_0\dots t_m$ automatically divides $\pi$. Sometimes it is convenient to view $t_0$ as a dummy variable. For example, one can represent $R\{t_1\.t_n\}$ as $R_{n,0}/(t_0-a)$ with $0\neq a\in \pi R$.

(ii) Berkovich considered the case when $\gtX$ is distinguished and one can take $a$ to be a uniformizer. Note that this happens if and only if $\gtX$ is distinguished and regular. Such formal schemes are called $\pi$-distinguished in \cite[Definition~2.1.1]{berform}.

(iii) Using desingularization of formal schemes, Berkovich proved in \cite[Theorem~2.1.2]{berform} that after modifying $\gtX$ and extending $R$, any reduced formal $R$-scheme $\gtX$ can be made (a) $\pi$-distinguished, (b) semistable. Our version of semi\-stable reduction asserts that both conditions can be achieved simultaneously.
\end{rem}

\begin{lem}\label{distlem}
Assume that $R$ is a complete discrete valuation ring with residue field $k$ of characteristic zero and $\gtX$ is a regular special formal $R$-scheme. Then $\gtX$ is strictly distinguished if and only if the special fiber $\gtZ=\gtX\otimes_R k$ is an snc divisor and the closed fiber $\gtX_s$ is a divisor.
\end{lem}
\begin{proof}
In fact, this is a particular case of \cite[Proposition~2.1.3]{berform}, where an analogous description of $\pi$-distinguished formal schemes is given. For completeness, we outline the argument.

The direct implication is clear, so assume that $\gtZ$ is snc and $\gtX_s$ is a divisor, and let us prove that $\gtX$ is strictly distinguished. We can work locally at a closed point $\gtx\in\gtX_s$. By our assumption there exists a regular family of parameters $t_0\.t_n\in\calO_{\gtX,\gtx}$ such that $\gtZ=V(\pi)$ is given by the vanishing of $t_0\dots t_l$ for $0\le l\le n$. Shrinking $\gtX$ we can assume that $t_i$ are global functions, and multiplying $t_0$ by a unit we can assume that $\pi=t_0\dots t_l$. Since $\gtX_s$ is a subdivisor of $\gtZ$, we can also assume that $\gtX_s$ is given by the vanishing of $t_0\dots t_m$ for $0\le m\le l$. Thus, we obtain a morphism $$\phi\:\gtX\to\gtY=\Spf\left(R_{n,m}/(t_0\dots t_l-\pi)\right),$$ and it suffices to prove that $\phi$ is \'etale at $x$.

Since $t_0\dots t_m$ generates ideals of definition of $\gtX$ and $\gtY$, the special morphism $\phi$ is of finite type. Thus, it suffices to prove that $h\:\calO_{\gtY,\gty}\to\calO_{\gtX,\gtx}$ is flat and unramified, where $\gty=\phi(\gtx)$. Since $m_\gtx$ and $m_\gty$ are generated by $t_0\.t_n$, the fiber of $h$ is the finite field extension $k(\gtx)/k$. Using that $k$ is of characteristic zero, we obtain that $\phi$ is unramified at $x$. To show that $h$ is flat, it suffices to show that its completion $\hath\:\hatcalO_{\gtY,\gty}\to\hatcalO_{\gtX,\gtx}$ is flat. The homomorphisms $\bfQ\into k\into k(\gtx)$ are formally smooth, hence we can find compatible fields of coefficients $k\into\hatcalO_{\gtY,\gty}$ and $k(\gtx)\into\hatcalO_{\gtX,\gtx}$. Since both complete rings are regular with parameters $t_0\.t_n$, we then have $k\llbracket t_0\.t_n\rrbracket\toisom\hatcalO_{\gtY,\gty}$ and $k(\gtx)\llbracket t_0\.t_n\rrbracket\toisom\hatcalO_{\gtX,\gtx}$. In particular, $\hath$ is a base change of $k\into k(\gtx)$, so $\hath$ is flat.
\end{proof}

\subsubsection{$d$-regularization}\label{redsec}
Given a reduced qe formal scheme $\gtX$ with a global function $f\in\Gamma(\calO_\gtX)$ and a number $d>0$, let $\gtX[f^{1/d}]^\nor$ denote the normalization of $\Spf_\gtX(\calO_\gtX\{f^{1/d}\})$. It follows from Lemma~\ref{distlem} that Theorem~\ref{stabth} is equivalent to the following: if $\gtX$ is a reduced special formal $R$-scheme then there exists a $(\gtX_\sing\cup\gtZ)$-supported blow up $\gtX'\to\gtX$ and $d>0$ such that $\gtX'[\pi^{1/d}]^\nor$ is regular and has an snc special fiber. The latter statement makes sense for arbitrary qe formal schemes, and we will prove it in that generality. Moreover, we will work with invertible ideals instead of global functions.

Assume that $\gtX$ is a reduced formal scheme, $\gtZ$ is a closed formal subscheme and $d>0$ is a number invertible on $\gtX$. We say that the pair $(\gtX,\gtZ)$ is {\em $d$-regular} if there exists an open covering $\gtX=\cup_i\gtX_i$ such that each $\gtZ_i=\gtZ\times_\gtX\gtX_i$ is principal, say $\gtZ_i=V(f_i)$, the formal schemes $\gtX'_i=\gtX_i[f_i^{1/d}]^\nor$ are regular, and the formal subschemes $\gtZ'_i=V(f_i^{1/d})$ are snc divisors. The definition is independent of the choice of generators $f_i$ because $d$ is invertible on $\gtX$. Thus, Theorem~\ref{stabth} will be proven once we establish the following $d$-regularization result:

\begin{theor}\label{dregth}
Assume that $\gtX$ is a reduced qe formal scheme of characteristic zero and $\gtZ$ is a closed formal subscheme. Then there exists a number $d>0$ and a blow up $\gtX'\to\gtX$ with center supported on $\gtX_\sing\cup\gtZ$ and such that the pair $(\gtX',\gtZ\times_\gtX\gtX')$ is $d$-regular.
\end{theor}
\begin{proof}
Recall that $\gtT$-supported blow ups are preserved by compositions. Hence resolving the pair $(\gtX,\gtZ)$ by Theorem~\ref{catth} we can assume that $\gtX$ is regular and $\gtZ$ is strictly monomial. In this special case, the theorem will be proved by combinatorial methods later. In fact, we will deal with a slightly more general situation in Theorem~\ref{formmonprth} below.
\end{proof}

\subsection{$d$-regularization for cone fans}
Combinatorial $d$-regularization can be formulated in few equivalent languages, and we start with the language of \cite{KKMS}.

\subsubsection{Cone fans}
The classical combinatorial objects appearing in toroidal geometry are rational polyhedral cone complexes $\Sigma$ with integral structures $L$, i.e. $\Sigma$ is glued from finitely many rational polyhedral cones $\{\sigma\subset N_\sigma=\bfR^{n_\sigma}\}_{\sigma\in F}$ along face maps and each $N_\sigma$ is provided with a lattice $L_\sigma\subset\bfQ^{n_\sigma}$ in a way compatible with the face maps. For shortness, we call $\Sigma=(\Sigma,L)$ a {\em cone fan}.

\subsubsection{Projective subdivisions}
To simplify notation, if $\Sigma'\to\Sigma$ is a subdivision then we view functions $\Sigma\to\bfR$ also as functions on $\Sigma'$. Let $\lam\:\Sigma\to\bfR$ be a convex integral pl function, in the sense that its restrictions $\lam_\sigma$ onto the cones are convex integral pl functions. Then there exists a minimal subdivision $\Sigma_\lam\to\Sigma$ such that $\lam$ is a linear function on (the simplices of) $\Sigma_\lam$. Such a subdivision is called {\em projective} in \cite[Definition III.1.5]{KKMS}.

\subsubsection{Root covers}
Let $f\:\Sigma\to\bfR$ be an integral linear function in the sense that its restrictions $f_\sigma$ onto the cones are integral and linear. Given a number $d>0$ we define a new fan $\Sigma'=\Sigma[d^{-1}f]$ by changing the integral structure of $\Sigma$ as follows: $L'_\sigma=L_\sigma[d^{-1}f]$ is the lattice of points $l\in L_\sigma$ such that $d|f(l)$. In other words, we make the minimal change of the integral structure such that $d^{-1}f$ becomes an integral function.%One can think about $\Sigma$ as the alteration of $\Sigma$ obtained by

\subsubsection{Regularity}
We say that $\Sigma$ is {\em regular} if each $L_\sigma$ possesses a distinguished basis $e_1\.e_{n_\sigma}$ such that $\sigma$ is the standard cone $\prod_{i=1}^{n_\sigma}\bfR_{\ge 0}e_i$. An integral linear function $f$ on a regular $\Sigma$ is called {\em snc} if for each $\sigma$ with the distinguished basis $e_1\.e_n$ one has that $f_\sigma(e_i)\in\{0,1\}$ for $1\le i\le n$. Finally, we say that the pair $(\Sigma,f)$ is $d$-regular if $\Sigma[d^{-1}f]$ is regular and $d^{-1}f$ is snc on $\Sigma[d^{-1}f]$. Now we can reformulate the main combinatorial result of \cite{KKMS} in a way resembling the semistable reduction theorem.

\begin{theor}\label{confanth}
Assume that $\Sigma$ is a cone fan and $f\:\Sigma\to\bfR$ is a convex integral pl function. Then there exists a projective subdivision $\Sigma'\to\Sigma$ and $d>0$ such that $f$ is integral on $\Sigma'$ and $(\Sigma',f)$ is $d$-regular.
\end{theor}
\begin{proof}
Note that the class of projective subdivisions is preserved by compositions. So, replacing $\Sigma$ by $\Sigma_f$ we can assume that $f$ is integral. Furthermore, by using \cite[Theorem~I.10]{KKMS} we can further refine $\Sigma$ so that it becomes regular. By $P$ we denote the subset of $\Sigma$ given by $f=1$.

Let $\tilSigma$ be the cone subfan of $\Sigma$ consisting of cones $\tilsigma\subset\Sigma$ such that the set $P\cap\tilsigma$ is bounded; this happens if and only if $f$ does not vanish on the edges of $\tilsigma$. We first solve the problem for $\tilSigma$ and $f|_\tilSigma$. Note that $\tilP=\tilSigma\cap P$ is a polytope complex and $\tilL=L\cap\tilP$ is an integral structure on $\tilP$. By \cite[Theorem~III.4.1]{KKMS} there exists $d>0$ and a projective subdivision $\tilP'$ of $\tilP$ whose polytopes are $d^{-1}\tilL$-integral simplices of minimal possible volume. The subdivision $\tilP'\to\tilP$ induces a projective subdivision $\tilSigma'=\tilSigma_\tilh\to\tilSigma$ of the cone fans, where to each polytope $p\in\tilP'$ corresponds a cone $\sigma_p$ in $\tilSigma'$. It is easy to see that the condition that $p$ is a $d^{-1}\tilL$-integral simplex of minimal volume is equivalent to the condition that $\sigma_p$ is regular with respect to the integral structure of $\tilSigma'[d^{-1}f]$, see \cite[p. 106-107]{KKMS}. So, $\tilSigma'$ is regular. In addition, if $e_1\.e_n$ is the distinguished basis of $\sigma_p$ then $f(e_i)>0$ by the definition of $\tilSigma$, and $f(e_i)\in d\bfZ$ by the definition of $L[d^{-1}f]$. By the minimality of the volume, we then have $f(e_i)=d$.

We extend the subdivision $\tilSigma_\tilh\to\tilSigma$ to a subdivision $\Sigma'=\Sigma_h\to\Sigma$ as follows. Any cone $\sigma$ in $\Sigma$ splits into the product of the standard cones $\tilsigma\times\osigma$, where $\tilsigma$ is in $\tilSigma$ and $f=0$ on $\osigma$. Pulling back $\tilh_\tilsigma$ to a function $h_\sigma$ for each $\sigma$, we obtain a convex integral pl function $h$ on $\Sigma$, and it is easy to see that the preimage of $\sigma$ in $\Sigma'$ is the product of $\osigma$ with the preimage of $\tilsigma$ in $\tilSigma'$. Thus, each cone of $\Sigma'$ splits into the product of a cone of $\tilSigma'$ and a cone $\osigma$. Also, any edge of $\Sigma'$ lies either in $\tilSigma'$ or in $\Sigma\setminus\tilSigma$, and hence its basic integral vector satisfies $f(e)=d$ or $f(e)=0$, accordingly. Since $(\tilSigma',f|_{\tilSigma'})$ is $d$-regular we obtain that $(\Sigma',f)$ is also $d$-regular.
\end{proof}

\subsection{Translation to fans}

\subsubsection{Kato's fans}
A more economical way to represent a cone fan $\Sigma$ is by use of the associated Kato's fan consisting of the set $F=F_\Sigma$ of cones of $\Sigma$ with the sheaf of monoids $\calM_F$ such that $\calM_{F,x}=L_x\cap\sigma_x$. We will simply call $F=(F,\calM_F)$ a {\em fan}. In this paper we only consider saturated fans, so the word saturated will often be omitted.

We refer to \cite[\S3.1]{ILO} for the definition of the category of saturated fans; it is easy to see that it is equivalent to the category of cone fans. In addition, for a fixed field $k$ with $S=\Spec(k)$, both categories are equivalent to the category of toric $k$-varieties with toric morphisms. As in \cite[\S3.2]{ILO}, we will denote the equivalences $\Sigma\mapsto S[\Sigma]$ and $F\mapsto S[F]$.

\subsubsection{Invertible ideals}
An ideal $I\subseteq\calM_F$ is called invertible if each $I_x$ is generated by an element $f_x\in\calM_{F,x}$. Since $\calM_{F,x}$ is a sharp monoid, $f_x$ is uniquely determined and this implies that these elements glue to a global section $f\in\Gamma(I)$. In particular, $I=(f)$ is principle. Note that principle ideals correspond to non-negative integral functions on cone fans.

\subsubsection{Blow ups}\label{subdsec}
The blow up of a fan $F$ along an ideal $I\subseteq\calM_F$ is naturally defined using charts, e.g. see \cite[\S3.2]{ILO}. We also recall that ideals in $\calM_F$ correspond to toric ideals on $S[F]$ bijectively and this functor takes blow ups of $F$ to toric blow ups of $S[F]$, see \cite[Lemma~3.2.17]{ILO}. By \cite[Chapter~I]{KKMS}, for any projective subdivision $\Sigma'\to\Sigma$, the corresponding morphism of toric varieties $S[\Sigma']\to S[\Sigma]$ is a toric blow up, and so the corresponding morphism of fans $F'\to F$ is also a blow up. We do not explore the question whether the ideal defining the blow up can be chosen canonically.

\subsubsection{$d$-regularization for fans}
We say that a fan $F$ is {\em regular} if each stalk $\calM_{F,x}$ is a free monoid $\bfN^{n_x}$. Assume $F$ is regular, then we call an ideal $I\subseteq\calM_F$ {\em snc} if each $I_x$ is generated by an element of the form $\sum_{i=1}^n a_ie_i$, where $e_1\. e_n$ form the basis of $\calM_{F,x}$ and $a_i\in\{0,1\}$. Assume that $I$ is invertible. Then $I=(f)$ and we obtain a canonical morphism $F\to\Spec(\bfN)$ taking the generator of $\bfN$ to $f$. Let $F[d^{-1}f]^\sat$ be the saturated base change of $F$ with respect to the map $\Spec(\bfN)\to\Spec(\bfN)$ corresponding to the homomorphism $d\:\bfN\to\bfN$. This construction corresponds to the construction of $\Sigma[d^{-1}f]$ in the language of cone fans, and it makes the analogy with the semistable reduction even more direct. We say that a pair $(F,I)$ is {\em $d$-regular} if $I$ is invertible, $F[d^{-1}I]^\sat$ is regular and $d^{-1}I$ is snc on $F[d^{-1}I]^\sat$. Theorem~\ref{confanth} can now be reformulated as follows:

\begin{theor}\label{fanth}
For any fan $F$ with an ideal $I\subseteq\calM_F$ whose stalks are non-empty there exists a blow up $h\:F'\to F$ and a number $d>0$ such that the pair $(F',h^{-1}I)$ is $d$-regular.
\end{theor}

\subsection{$d$-regularization for log regular log schemes}\label{uniflogsec}

\subsubsection{Log regular log schemes}
We refer to \cite[Section 2]{Kato-toric} for the definition of log regular log schemes. If $X$ is a log regular log scheme then the set $D=X(0)$ of all points $x\in X$ with $\ocalM_x=1$ is a divisor and $\calM_X$ is determined by $D$ via $\calM_X=\calO_X\cap i_*(\calO_U^\times)$, where $U=X\setminus D$ and $i\:U\into X$. In particular, one can represent log regular log schemes by pairs $(X,D)$, and the typical example is when $X$ is regular and $D$ is an snc divisor. Furthermore, if $(X,D)$ is log regular then the following conditions are equivalent: (a) $X$ is regular, (b) $X$ is regular and $D$ is snc, (c) all stalks $\ocalM_x$ are free monoids.

\subsubsection{Monoidal ideals}
An ideal $\calI$ on a log scheme $(X,D)$ is called {\em monoidal} if it is of the form $I\calO_X$ for an ideal $I\subseteq\calM_X$. In this case, we say that $\Spec(\calO_X/\calI)$ is a monoidal subscheme of $X$. We refer to \cite[Section 4]{Niziol} for the definition of the log blow up along a monoidal ideal $\calI$. When $(X,D)$ is log regular, it reduces to the normalized blow up along $\calI$ on the level of schemes, see \cite[Proposition~4.3]{Niziol}.

\subsubsection{Fans of log regular schemes}
To any log regular log scheme $(X,D)$ one associates a fan $F=F(X,\calM_X)$ as follows: $F=\cup_{n\in\bfN} F_n$ is the set of generic points of the log strata $X(n)$ of $(X,D)$ (i.e. the strata where the rank of $\ocalM_X$ equals to $n$) and $\calM_F=\ocalM_X|_F$. Monoidal ideals of $(X,D)$ correspond bijectively to ideals of the fan, and log blow ups of $(X,D)$ correspond to blow ups of fans: $F({\rm LogBl_\calI}(X))=\Bl_I(F)$, where $I\subseteq\calM_F$ corresponds to $\calI$, see \cite[Theorem~4.7]{Niziol}. Naturally, we say that $(F,I)$ is the fan of the pair $((X,D),\calI)$. Theorem~\ref{fanth} then implies the following analogue:

\begin{theor}\label{monprth}
Assume that $(X,D)$ is a log regular log scheme and $\calI$ is a monoidal ideal. Then there exists a log blow up $f\:(X',D')\to(X,D)$ and a number $d>0$ such that the fan of the pair $((X',D'),f^{-1}\calI)$ is $d$-regular.
\end{theor}

\begin{rem}
If $d$ is invertible on $X$ then the fan of a pair $((X,D),\calI)$ is $d$-regular if and only if the pair $(X,\calI)$ is $d$-regular. Hence, Theorem~\ref{monprth} implies the scheme-theoretic analogue of Theorem~\ref{dregth}.
\end{rem}

\subsection{The case of formal log schemes}
It remains to extend the above theory from schemes to formal schemes. The notions of \S\ref{uniflogsec} are compatible with regular morphisms and hence can be extended to other geometric categories. For concreteness, we work out the case of formal schemes, which is used to prove Theorem~\ref{dregth}.

Let $(\gtX,\calM_\gtX)$ be a qe formal log scheme. If $\gtX=\Spf(A)$ is affine then $\Gamma(\calM_\gtX)\to A$ gives rise to an affine log scheme $X=(X,\calM_X)$ with $X=\Spec(A)$, and we say that $(\gtX,\calM_\gtX)$ is log regular if $(X,\calM_X)$ is. Since $\gtX$ is qe, formal localizations are regular morphisms and it follows that affine formal subschemes of $(\gtX,\calM_\gtX)$ are log regular too. In particular, the notion of log regular formal log schemes globalizes.

In the same fashion, one defines the log blow up $\gtX'={\rm LogBl}_\calI(\gtX)$ along a monoidal ideal: choose an appropriate affine covering of $\gtX$ by $\gtX_i=\Spf(A_i)$ and glue the formal completions of the log blow ups of $X_i=\Spec(A_i)$.

The only place where one should exercise some care is with extending the notion of fans. One cannot define log strata $\gtX(n)$ by gluing formal subschemes of $\gtX_i$ corresponding to $X_i(n)$, but, fortunately, it suffices to use the closures $X(\ge\! n)$ of $X(n)$ instead. Clearly $X_i(\ge\! n)$ are closed subschemes compatible with the morphisms $\Spec(A_{ijk})\to X_i$, where $\gtX_i\cap\gtX_j=\cup_k\Spf(A_{ijk})$ is an open covering. Therefore, the corresponding closed formal subschemes $\gtX_i(\ge\! n)$ of $\gtX_i$ glue to a closed formal subscheme $\gtX(\ge\! n)$ of $\gtX$, thus defining the locus of $\gtX$ where $\rk(\ocalM_\gtX)\ge\! n$. We define the fan of $\gtX$ by $F=\cup_{n\in\bfN}F_n$, where $F_n$ is the set of irreducible components of $\gtX(\ge\! n)$ which are not contained in $\gtX(\ge\! n+1)$. Now, all basic results about log blow ups, including compatibility with blow ups of fans, can be extended to formal schemes in the obvious way, and the proofs reduce to choosing an affine covering and using the analogous results for schemes. In particular, we obtain the following extension of Theorem~\ref{monprth}.

\begin{theor}\label{formmonprth}
Assume that $(\gtX,\gtD)$ is a log regular formal log scheme and $\gtI$ is a monoidal ideal. Then there exists a log blow up $f\:(\gtX',\gtD')\to(\gtX,\gtD)$ and a number $d>0$ such that the fan of the pair $((\gtX',\gtD'),f^{-1}\gtI)$ is $d$-regular. In particular, if $\gtX$ is of characteristic zero then the pair $((\gtX',\gtD'),f^{-1}\gtI)$ is $d$-regular.
\end{theor}

As we saw earlier, Theorem~\ref{formmonprth} implies Theorem~\ref{dregth}, which in its turn implies Theorem~\ref{stabth}.

\appendix

\section{Motivation for introducing $B$-schemes}\label{bmotsec}
In this appendix we advertise $B$-schemes and try to explain why their definition in desingularization context is natural. Also, we explain the logic of other definitions, for example, of the complete transform.

\subsection{Embedded desingularization and the boundary}\label{embsec}
At least since the great work \cite{Hironaka} of Hironaka, it is a common knowledge that in the embedded desingularization one should give a special treatment to the exceptional divisor accumulated through the blow up sequence. It is also common to call this divisor the boundary, and, indeed, in some aspects it behaves like a boundary. Thus, for the sake of mastering an inductive desingularization procedure one should consider triples $(X,E,Z)$ even if one starts with an empty $E$.

Also, it is now a standard observation that although the support of $E$ is an snc divisor, at least at some stages of the algorithm one should provide $E$ with the finer structure of splitting to regular components and ordering them (see also \S\ref{strsec}). Very naturally, both tasks are accomplished by the history of the desingularization process, and, after adding the history function, $E$ becomes an snc boundary in our sense. In principle, there are finer versions of the algorithm that involve less history data, but in this paper we choose to work with the total history and defined the $B$-schemes accordingly. Note also that the rule of forming the new boundary from the old one is encoded in the complete transform.

\subsection{Non-embedded desingularization and the boundary}\label{nonsec}
A common approach to building a non-embedded desingularization of a scheme $Z$ is to embed it into a regular ambient scheme $X$ and to apply embedded desingularization to $(X,\emptyset,Z)$. For example, this gives a non-strong desingularization of equidimensional varieties of characteristic zero in many works, including \cite{Jarek} and \cite{Kollar-resolution}. As an output one gets a permissible blow up sequence $f\:X'\longto X$ such that $\tilf\:Z'=f^!(Z)\longto Z$ is a desingularization and $Z'$ has simple normal crossings with the boundary $E'$ of $f$. In particular, in addition to getting a regular $Z'$ one automatically obtains that the exceptional divisor $E'|_{Z'}$ of $\tilf$ is snc. Moreover, if we start with an arbitrary initial snc boundary $E$ such that $D:=E|_Z$ is a divisor and resolve $(X,E,Z)$ by $f\:X'\longto X$, then its restriction $\tilf\:Z'\longto Z$ also makes $D\times_ZZ'$ a strictly monomial divisor whose reduction has simple normal crossings with the exceptional divisor.

To summarize, the embedded desingularization makes more than just desingularizing the embedded scheme $Z$; it desingularizes the embedded pair $(Z,E|_Z)$, and this is non-trivial even when $E|_Z=\emptyset$. This observation indicates that even for non-embedded desingularization one implicitly deals with boundaries, and the natural problem one solves, even without planning to, is to desingularize a pair $(Z,B)$ where $B=E|_Z$ is a boundary. If one works with a total history then $E$ is a set of divisors, and hence $B$ should at least be a set of divisors rather than a single divisor. Moreover, it is natural to restrict $E$ onto closed subschemes without excluding non-divisorial cases, and this directly leads to our definition of boundaries.

\subsection{$B$-permissibility}
As we discussed in \S\ref{embsec}, one usually uses permissible blow up sequences of an ambient $B$-scheme $(X,E)$. When constructing a strong embedded resolution of $Z\into X$ one only blows up centers on the strict transform of $Z$, e.g. see \cite[12.2]{BMcanonical}. In this case the blowing up sequence $f\:X'\longto X$ is determined by the blow up sequence $g\:Z'\longto Z$ of strict transforms. Namely, $f$ is the pushforward of $g$. It is then natural to encode all data we define on $X$ in terms of  a data defined on $Z$. This naturally leads to the definitions of boundaries and $B$-permissible blow ups. In particular, Lemma \ref{permisslem} says that $f$ is $E$-permissible if and only if $g$ is $E|_Z$-permissible.

\subsection{$B$-schemes}
In order to discuss desingularization of pairs $(Z,B)$ it looks natural to link them into a single object, and the fact that such a pair can be interpreted as a log scheme of a special type gives a strong indication that this definition makes sense. We will discuss below two situations where the use of $B$-schemes seems to be very natural.

\begin{rem}
It is an interesting question whether more general log schemes can be useful for desingularization theory.
\end{rem}

\subsection{Redundancy of blow ups}
There are two possibilities to work with blow ups. In \cite{temdes}, by a blow up one means a morphism $X'\to X$ that is isomorphic to a blow up along some center, while in \cite{nemb} and in this paper the center of a blow up is a part of the data, and hence blow ups are enriched morphisms. The latter is crucial in order to have strict and principal transforms. There are also various examples (obvious and not) of different blow up sequences that produce the same morphism but induce different transforms, see \cite[3.33]{Kollar-resolution}. However, if we consider a blow up of schemes $f\:X'\to X$ as a $B$-blow up of $B$-schemes $(X',E')\to(X,\emptyset)$, then the center is almost always determined by the $B$-blow up, and similarly for the blow up sequences from \cite{Kollar-resolution}. The only small redundancy with $B$-blow ups was described in Remark \ref{blowrem}(i), and even that could be avoided by using the history function for ordering, i.e. by ordering components by the natural numbers so that empty components are allowed.

\subsection{Transforms}
Principal (weak or controllable) transforms of closed subschemes or ideals are commonly used in embedded desingularization. The idea is to split off some multiples of the exceptional divisors from pullbacks of the ideal until nothing is left.

However, in general there is no morphism $(X',f^\rhd(B))\to(X,B)$, and this issue is resolved by introducing the complete transform of $B$, which is the minimal natural increment of $f^\rhd(B)$ such that there is a natural morphism of $B$-schemes $(X',f^\circ(B))\to(X,B)$. Note also that $|f^\circ(B)|=|f^!(|B|)|\cup|E_f|$ and so the complete transform keeps at least set-theoretical information about the old boundary and the boundary of $f$.

\subsection{On strict desingularization from \cite{temdes}}\label{strsec}
If $X$ is a qe scheme of characteristic zero and $B$ is a divisorial boundary on $X$ then by Theorem~\ref{Bth} there exists a $B$-blow up sequence $\calF(X,B)\:(X',B')\to(X,B)$ such that $X'$ is regular, $B'$ is snc and $\calF(X,B)$ only modifies the locus where $X$ is not regular or $B$ is not snc. A non-functorial version of this theorem was established in \cite{temdes} under the name of semi-strict desingularization. It was formulated in terms of the underlying divisors, i.e. it addressed $D=|B|$ rather than $B$. Moreover, one similarly defined in \cite{temdes} strict desingularization of a pair $(X,D)$ as a blow up sequence that resolves $X$, makes $D$ normal crossings and only modifies the locus where $X$ is not regular or $D$ is not normal crossings. It might look surprising, but strict desingularization does not exist even for algebraic surfaces, as one can see in the classical example of Whitney umbrella discussed in Example \ref{exam} below.

\begin{rem}\label{strrem}
(i) Existence of strict desingularization of varieties was incorrectly proved in \cite[Theorem~2.2.11]{temdes}. The mistake in that proof was in claim (i) and it is due to the fact that the number of formal brunches through a point of $D$ is not Zariski semi-continuous, unlike the number of irreducible components. This should be corrected by replacing strict desingularization and formal branches by semi-strict desingularization and irreducible components in the formulation and proof of \cite[Theorem~2.2.11]{temdes}. Actually this was in the original argument I heard from Bierstone-Milman! The correction does not affect anything else in \cite{temdes}.

(ii) Functorial semi-strict desingularization does not exist. Indeed, such desingularization would automatically be strict due to the fact that normal crossings and strict normal crossings loci are indistinguishable in the \'etale topology. The best one can hope for is a semi-strict desingularization that it is only functorial with respect to morphisms that preserve the number of irreducible components through any point. This is not an \'etale invariant and missing this subtlety lead me to the mistake in \cite{temdes}. The semi-strict desingularization obtained from $\calF$ is slightly weaker: it depends on an ordering of the components and is only functorial with respect to smooth morphisms that respect the order.

(iii) The above subtleties do not show up when one works with boundaries $D=\{D_1\. D_n\}$, snc loci of boundaries, and strict regular morphisms. For example, strict regular morphisms respect the $k$-multiple locus $D(k)$ that consist of points contained in exactly $k$ boundary components. Thus, the proof of \cite[2.2.11]{temdes} applies to the $B$-variety $(Z,D)$ and provides a strong functorial desingularization of $B$-varieties. This gives another proof of Theorem \ref{bvardesth}, which has the advantage of being independent of the order of the components. In particular, the latter proof also applies to unordered boundaries.
\end{rem}

\begin{exam}\label{exam}
Take the Whitney umbrella $D\subset Z=\bfA_k^3=\Spec(k[x,y,z])$ given by $x^2=zy^2$. It has the so-called pinching point singularity at the origin $O$. Clearly, $D$ is smooth outside of the $z$-axis $C$, is not monomial at $O$, and is normal crossings but not snc at any point of $C_0=C\setminus\{O\}$. We will show that any modification $D'\to D$ with normal crossings $D'$ has to resolve $D$ at the generic point of $C$. In particular, it cannot preserve the whole normal crossings locus of $D$.

A nice topological argument is given by J. Kollar at \cite{Kollar-semilog}. Here is a slightly different argument that applies over any field. An easy computation shows that the preimage of $C_0$ in the normalization of $D_0=D\setminus\{O\}$ is an irreducible \'etale double covering $\tilC_0$ of $C_0$ whose extension to a covering $\tilC\to C$ with a normal $\tilC$ ramifies over $O$. Assume now that $D'\to D$ is a modification inducing an isomorphism over $D_0$, and let $\tilD'\to D'$ be its normalizations. Clearly, $C_0$ embeds into $D'$, its Zariski closure $C'\subset D'$ is isomorphic to $C$, and the preimage of $C'$ in $\tilD'$ is an irreducible double covering that ramifies over $O':=C'\setminus C_0$. It follows easily that $D'$ is not normal crossings at $O'$.
\end{exam}

\subsection{Non-existence of normal crossings compactifications}
The same example with a pinching point shows that one cannot replace strict normal crossings with normal crossings in Theorem~\ref{compacth}.

\begin{exam}\label{compexam}
(i) Let $C,D,Z,D_0,C_0$ be as in Example~\ref{exam}. Then $D_0$ is a normal crossings surface such that any compactification $\oD$ of $D_0$ is not normal crossings, as can be shown by precisely the same argument. Indeed, the closure $\oC$ of $C_0$ in $\oD$ admits a morphism from $C$, so let $\oO\in\oC$ be the image of $O\in C$. Then the preimage of $\oC$ in the normalization of $\oD$ is a double covering which ramifies over $\oO$. Therefore, $\oD$ is not normal crossings at $\oO$.

(ii) Consider the smooth variety $Y=Z\setminus\{O\}$ with the normal crossings divisor $D_0$ and set $Y_0=Y\setminus D_0$. We claim that there is no regular compactification $\oY$ of $Y$ such that $\oY\setminus Y_0$ is a normal crossings divisor. Indeed, if such a compactification exists then the closure $\oD$ of $D_0$ in $\oY$ is a normal crossings divisor, contradicting part (i).
\end{exam}

\bibliographystyle{amsalpha}

%\bibliography{embedded}

\end{document}